\documentclass[12pt,oneside,reqno]{amsart}
\usepackage{graphicx}
\usepackage{dcolumn}
\usepackage{amsmath,amsthm,amssymb}
\usepackage{todonotes}
\usepackage{caption}

\newtheorem{theorem}{Theorem}
\theoremstyle{remark}
\newtheorem{rem}{Remark}
\theoremstyle{remark}
\newtheorem{step}{Step}

\DeclareMathOperator{\sech}{sech}

\begin{document}

\title[Yang-Mills field on an asymptotically  hyperbolic space]{Global dynamics of a Yang-Mills field \\on an asymptotically  hyperbolic space}
\author{Piotr Bizo\'n}
\address{Institute of Physics, Jagiellonian
University, Krak\'ow, Poland\\ and
Max Planck Institute for Gravitational Physics (A. Einstein Institute),
Golm, Germany}
\email{piotr.bizon@aei.mpg.de}

\author{Patryk  Mach}
\address{Institute of Physics, Jagiellonian
University, Krak\'ow, Poland}
\email{patryk.mach@uj.edu.pl}
\begin{abstract}
We consider a spherically symmetric (purely magnetic) $SU(2)$ Yang-Mills  field propagating on an ultrastatic  spacetime with two asymptotically hyperbolic regions connected by a throat of radius $\alpha$. Static solutions in this model are shown to exhibit an interesting bifurcation pattern in the parameter $\alpha$. We relate this pattern  to the Morse index of the static solution with maximal energy. Using a hyperboloidal approach to the initial value problem, we describe the relaxation to the ground state  solution for generic initial data and unstable static solutions for  initial data of codimension one, two, and three.
\end{abstract}
\maketitle

\section{Introduction and setup}
The evolution of a Yang-Mills  field in a four-dimensional Minkowski spacetime is rather uneventful; all  solutions starting from smooth initial data at $t=0$  remain smooth~forever \cite{em} and decay to zero as $t\rightarrow\infty$ \cite{ch,bcr}. In a globally hyperbolic four-dimensional curved spacetime the evolution can be more intricate: although singularities cannot develop \cite{cs}, there may exist nontrivial stationary attractors
    giving rise to critical phenomena at the  boundaries of their basins of attraction, as was shown for the Schwarzschild background \cite{brz}.

    Here we consider the evolution of a spherically symmetric $SU(2)$ Yang-Mills field on an ultrastatic  spacetime $(\mathcal{M}, g)$ with the manifold $\mathcal{M}=\{(t,r)\in \mathbb{R}^2, (\vartheta,\varphi)\in S^2\}$ and metric
\begin{equation}\label{g}
g=-dt^2+dr^2+\alpha^{2} \cosh^2{r}\,d\omega^2\,,
\end{equation}
where $d\omega^2=d\vartheta^2+\sin^2{\!\vartheta}\, d\varphi^2$ is the round metric on the unit 2-sphere. The hypersurfaces $t=\text{const}$  are three-dimensional asymptotically hyperbolic cylinders that are symmetric under the reflection $r\rightarrow -r$. The neck, $r=0$, of this hyperbolic wormhole is a minimal surface of area $4\pi \alpha^2$. The Ricci scalar of \eqref{g} is equal to $R(g)=-6+(2+2\alpha^{-2}) \sech^2{r}$.

 We are interested in an $SU(2)$ Yang-Mills field propagating on $(\mathcal{M}, g)$.  The gauge potential $A_{\mu}=A_{\mu}^a \tau_a$ takes values in the Lie algebra $su(2)$, where the generators $\tau_a$   satisfy $[\tau_a,\tau_b]=i \epsilon_{abc} \tau_c$. In terms of the Yang-Mills field tensor, $F_{\mu\nu}=\nabla_{\mu} A_{\nu} - \nabla_{\nu} A_{\mu} +[A_{\mu}, A_{\nu}]$, the lagrangian density reads
\begin{equation}\label{action}
  \mathcal{L} = \mathrm{Tr}\left(F_{\alpha\beta} F_{\mu\nu} g^{\alpha\mu} g^{\beta\nu}\right)\, \sqrt{-\mathrm{det}(g_{\mu\nu})}.
\end{equation}
For the  Yang-Mills potential we assume the  spherically symmetric purely magnetic ansatz
\begin{equation}\label{A}
  A = W(t,r)\, \eta +\tau_3 \cos{\vartheta} d\varphi\,,\quad\mbox{where} \quad \eta=\tau_1 d\vartheta+\tau_2 \sin{\vartheta} \,d\varphi\,,
\end{equation}
 which gives the Yang-Mills field tensor
\begin{equation}\label{F}
  F=\partial_{t} W dt\wedge \eta + \partial_r W dr\wedge\eta -(1-W^2)\, \tau_3 \,d\vartheta \wedge \sin{\vartheta} \,d\varphi\,.
\end{equation}
Note that the vacuum state $W = \pm 1$ is two-fold degenerate.
Inserting ansatz \eqref{F} into \eqref{action} we get the reduced lagrangian density
\begin{equation}\label{red_action}
 \mathcal{L}=-\frac{1}{2} \left(\partial_{t} W\right)^2 +\frac{1}{2} \left(\partial_{r} W\right)^2 +\frac{(1-W^2)^2}{4 \alpha^2 \cosh^2{r}} \,.
\end{equation}
  Hereafter it is convenient to define the constant $\ell$ by $\ell(\ell+1)=\alpha^{-2}$. Then,  the Euler-Lagrange equation derived from \eqref{red_action} reads
\begin{equation}\label{eqym}
  \partial_{tt} W=\partial_{rr} W+\frac{\ell(\ell+1)}{\cosh^2{r}} \, W (1-W^2)\,,
\end{equation}
and the  associated conserved energy is
\begin{equation}\label{energy-r}
  E=\frac{1}{2} \int_{-\infty}^{\infty} \left((\partial_t W)^2 + (\partial_r W)^2 + \frac{\ell(\ell+1)}{2\cosh^2{r}}\, (1-W^2)^2\right) \, dr\,.
\end{equation}
The Yang-Mills equation \eqref{eqym} is a bona fide $1+1$ dimensional semilinear wave equation for which
 it is routine to show that solutions  starting at $t=0$ from smooth finite-energy initial data remain smooth for all future times\footnote{As we have mentioned above,  global regularity holds for the Yang-Mills equation on any four dimensional globally hyperbolic spacetime \cite{cs}, however  the proof of this fact is highly nontrivial already in  Minkowski spacetime \cite{em}.}.

 The goal of this paper is to describe the asymptotic behaviour of solutions for $t\rightarrow \infty$.
 Due to the  dissipation of energy by dispersion,  solutions are expected to settle down to critical points of the potential energy,
  i.e. static
 solutions. Before studying the evolution, in the following two sections we describe the static sector of the model.
\section{Static solutions}
Time-independent solutions $W=W(r)$ of Eq.\eqref{eqym} satisfy the ordinary differential equation
\begin{equation}\label{eqs}
  W''+\frac{\ell(\ell+1)}{\cosh^2{r}}\, W (1-W^2)=0\,.
\end{equation}
Due to the  reflection symmetry $W\rightarrow -W$, all solutions (except the reflection invariant solution $W_*=0$) come in pairs. In the following, each pair $\pm W$  will be counted as one solution. Besides $W_*$, the second constant solution is  $W_0=1$ which is the ground state with zero energy. These two constant solutions will play the key role in our analysis\footnote{Note
that, in contrast to the Yang-Mills model in a flat space, $W_*$ is a smooth, finite energy solution; it is a hyperbolic analogue of Wheeler's ``charge without charge'' configuration \cite{w} with two magnetic charges of opposite signs $\pm 1$ sitting at $r=\pm \infty$.}.

We will show below that as  $\ell$ increases from zero, nonconstant  static solutions with finite energy
 bifurcate from $W_*$ at each positive integer value of $\ell$. These solutions are either odd or even so it is sufficient to analyze them for $r\geq 0$. The odd solutions are parametrized by $W'(0)$ and the even solutions are parametrized by $W(0)$.
 Near $r=\infty$ these solutions belong to the one-parameter family of solutions
\begin{equation}\label{c-inf}
  W(r)=W_{\infty}+\mathcal{O}(e^{-2r})\,,
\end{equation}
which are analytic in the parameter $W_{\infty}$ and $e^{-2r}$. Smooth solutions of Eq.~\eqref{eqs} satisfying \eqref{c-inf} will be referred to as regular solutions.
\begin{theorem}
Let $n$ be a nonegative integer. For $2n<\ell<2n+2$ there are exactly $n$ regular even solutions $W_{2n}(r)$  and for $2n+1<\ell<2n+3$ there are exactly $n$ regular odd solutions $W_{2n+1}(r)$ (the subscript of $W$  counts the number of zeros of $W(r)$ on the real line).
\end{theorem}
\begin{proof}
We consider only the odd solutions (the proof of existence of even solutions is analogous).
The odd solution with $b=W'(0)$ will be denoted by $W(r,b)$. Without loss of generality we assume that $b\geq 0$.
The proof, following the lines of the shooting argument given in \cite{b1}, proceeds in three steps:
\begin{step}[a priori global behavior]
 It follows from \eqref{eqs} that $W(r)$ cannot have  a maximum (resp.~minimum) for $W>1$ (resp.~$W<-1$), so once  $W(r,b)$   leaves the strip $|W|<1$, it cannot reenter it. Suppose that $|W(r,b)|<1$ for all $r$. We define the functional
\begin{equation}\label{Q}
  Q(r)=\cosh^2{r}\, W'^2-\frac{\ell(\ell+1)}{2} (1-W^2)^2\,.
\end{equation}
From Eq.~\eqref{eqs} we have
\begin{equation}\label{dQ}
Q'(r)=\sinh(2r) \, W'^2\,,
\end{equation}
hence $Q(r)$ is increasing for $r\geq 0$. It is easy to see that if $Q(r_0)>0$ for some $r_0$, then $|W(r_1,b)|=1$ for some $r_1>r_0$. Thus, $Q(r)<0$ for all $r\geq 0$, in particular  $Q(0)=2 b^2-\ell(\ell+1)<0$. Together with Eq.~\eqref{dQ} this implies that $\lim_{r\rightarrow\infty} Q(r)\leq 0$ exists, and therefore $\lim_{r\rightarrow \infty} Q'(r)=0$, which is equivalent to
$\lim_{r\rightarrow \infty} e^{r} W'(r)=0$. The existence of $\lim_{r\rightarrow\infty} Q(r)$, and \eqref{Q} imply in turn that  $W_{\infty}:=\lim_{r\rightarrow \infty} W(r)$ exists. Concluding, the solution for which $|W(r,b)|<1$  for all $r$ has the desired asymptotic behavior \eqref{c-inf}.
\end{step}
\begin{step}[behavior of solutions for small $b$]
 Inserting $W(r)=b w(r)$  into Eq.~\eqref{eqs} and taking the  limit $b\rightarrow 0$ we get
\begin{equation}\label{eq-v}
  w''+\frac{\ell(\ell+1)}{\cosh^2{r}}\, w =0,\quad w(0)=0, \quad w'(0)=1\,.
\end{equation}
The solution of  this limiting equation
 is given by the Legendre function
\begin{equation}\label{hyper}
w_L(r)=\tanh{r}\, {}_2 F_1\left(\frac{1-\ell}{2},\frac{\ell+2}{2},\frac{3}{2}; \tanh^2{r}\right)\,,
\end{equation}
which oscillates $n+1$ times around zero (i.e., has $n+1$ zeros) for $r>0$, where $2n+1$ is the smallest odd number less than $\ell$,  and then diverges to $(-1)^{n+1} \infty$ (or goes to $(-1)^{n+1}$, if $2n+3=\ell$). By continuity,  for $b$  small enough, $W(r,b)$ behaves in the same manner on compact intervals.
\end{step}
\begin{step}[shooting argument]
We define the set
\begin{equation*}\label{B1}
  B_1=\{
  b\,|\, W(r,b)\,\, \mbox{grows monotonically to}\,\, W(r_0,b)=1 \,\,\mbox{at some}\,\, r_0\}\,.
\end{equation*}
We know from Step 1 that the set $B_1$ is nonempty because all solutions $W(r,b)$ with $b>\sqrt{\ell(\ell+1)/2}$ belong to it. On the other hand, we know  from Step~2 that if $\ell>1$ then the set $B_1$ is bounded from below by a positive constant. Thus, $b_1=\inf B_1$  is strictly positive. The solution $W(r,b_1)$ cannot touch $W=1$ for a finite $r$ because the same would be true for nearby orbits, violating the definition of $b_1$.
Therefore, $|W(r,b)|<1$ for all $r$ and hence, by Step~1, $W(r,b_1)$ is the desired first regular odd solution with one zero. We denote it by $W_1$.

Next, consider the solution $W(r,b_1-\varepsilon)$  for a small positive $\varepsilon$. From the definition of $b_1$ it follows that $W(r,b_1-\varepsilon)$ attains a local maximum $W(r_0)<1$ for some $r_0$. It is not difficult to show that for sufficiently small $\varepsilon$ the solution $W(r,b_1-\varepsilon)$, after reaching the maximum at $r_0$, decreases monotonically to   $W(r_1)=-1$ for some $r_1>r_0$ (we omit the proof of this fact because it is very similar to the analogous proof given in \cite{b1}). This means that the set
\begin{align*}
& B_3=\{
  b\,|\, W(r,b)\,\, \mbox{grows monotonically to a maximum at some $r_0$ and } \\&\mbox{then monotonically decreases to} \,\, W(r_1,b)=-1\,\,\mbox{at some} \,\, r_1>r_0\}
\end{align*}
is nonempty. We know from Step~2 that if $\ell>3$, then $b_3=\inf B_3$ is strictly positive.  By the same argument as above, the solution $W(r,b_3)$ stays in the strip $|W|<1$ for all $r$ and gives the desired regular odd solution with three zeros. We denote it by  $W_3(r)$.
\end{step}
The subsequent solutions $W_{2n+1}(r)$ can be obtained by iterating the argument as long as $2n+1<\ell$.
\end{proof}
The first few solutions $W_n$  computed numerically are shown in Fig.~\ref{fig1}.

\begin{figure}[h]
\includegraphics[width=0.49\textwidth]{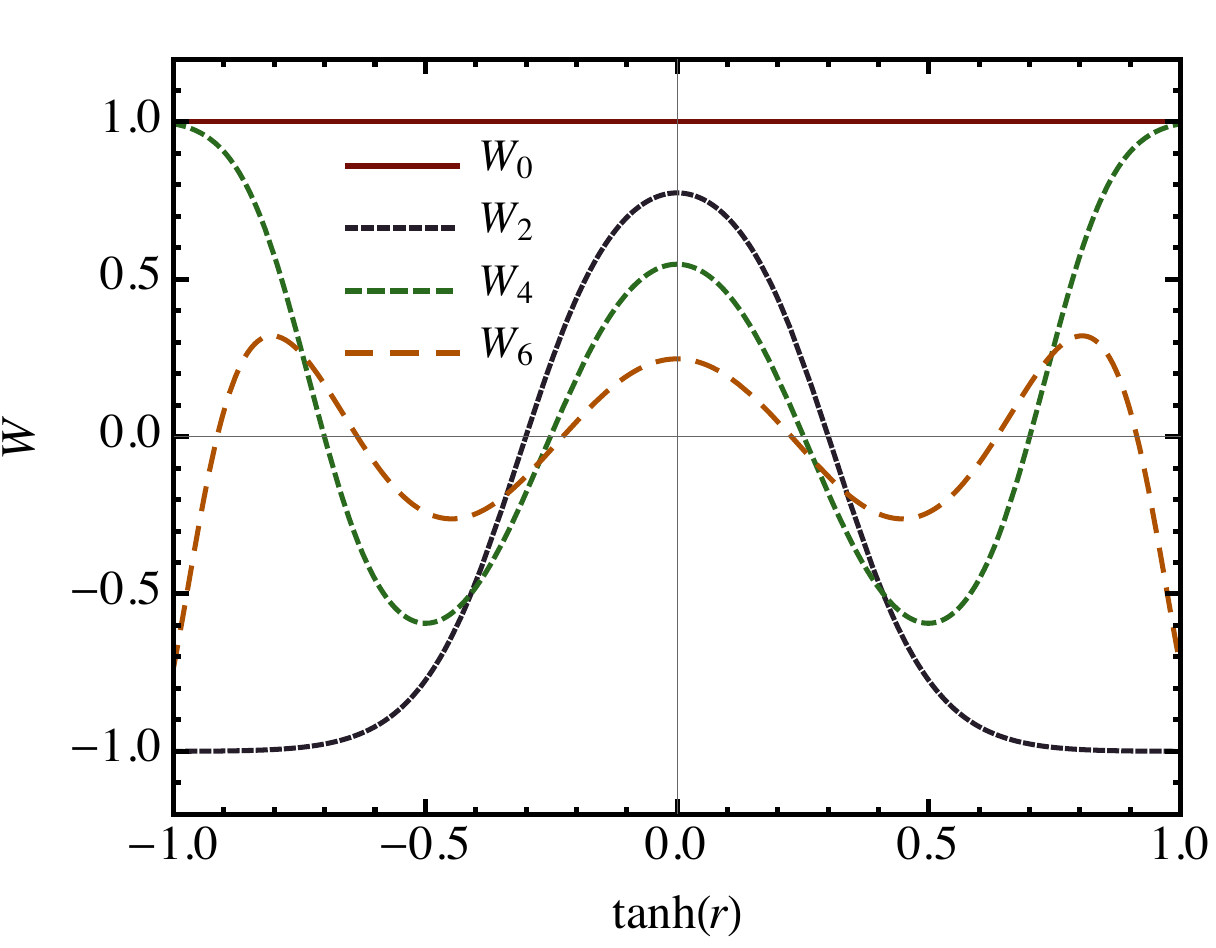}
\includegraphics[width=0.49\textwidth]{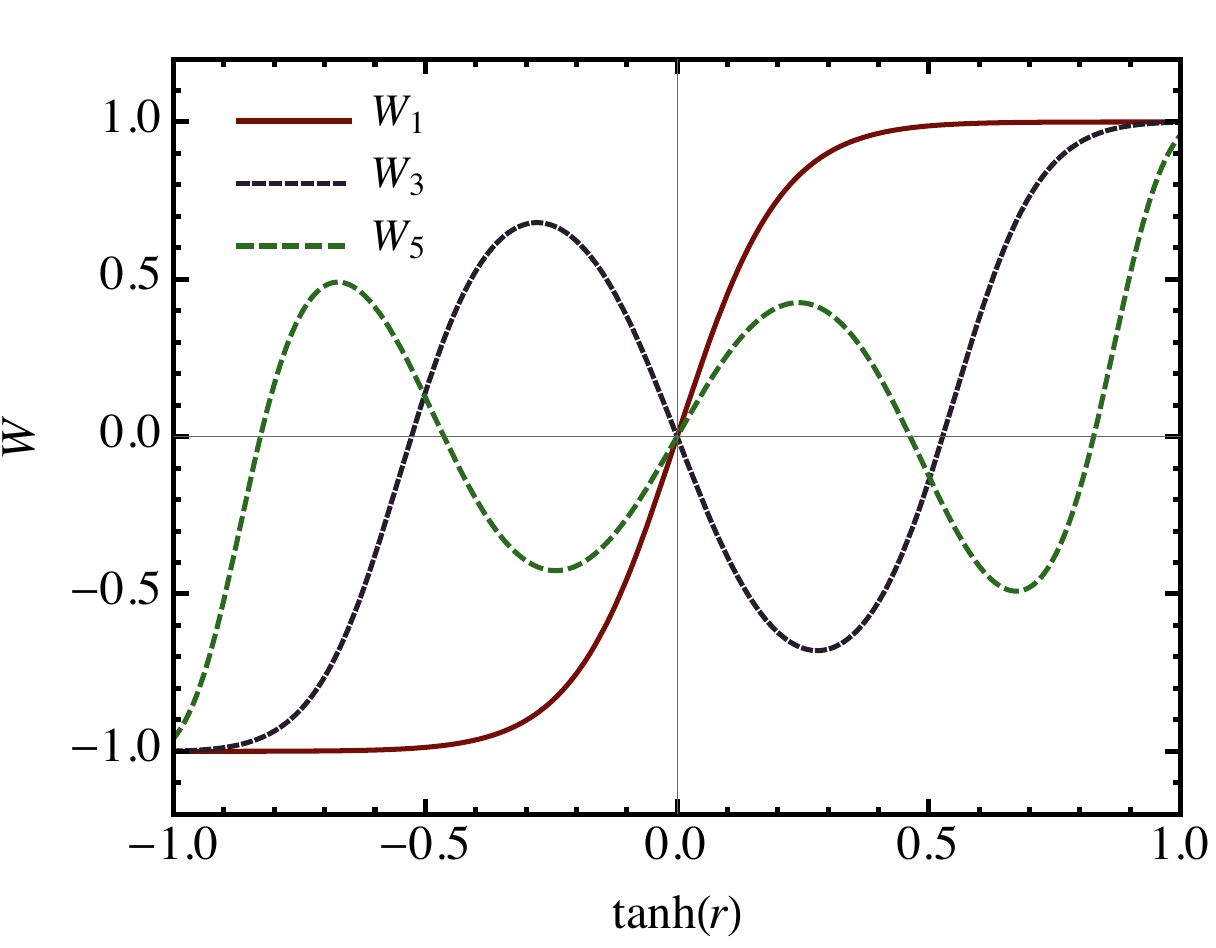}
\captionsetup{width=\textwidth}
\caption{\label{fig1}{\small Static solutions $W_n$ for $\ell=6.5$. The $r$-axis is compactified to the interval $[-1,1]$ by using $\tanh{r}$.}}
\end{figure}

We conclude this section with a few remarks concerning the properties of static solutions.
\begin{rem}
Multiplying Eq.~\eqref{eqs} by $W$ and integrating by parts we get the virial identity
    \begin{equation}\label{virial}
      \int_{-\infty}^{\infty} W'^2 dr= \ell(\ell+1)  \int_{-\infty}^{\infty} W^2 (1-W^2) \sech^2{r} \,dr.
    \end{equation}
    Inserting this into Eq.~\eqref{energy-r} gives for static solutions
\begin{equation}\label{e-virial}
      E= \frac{\ell(\ell+1)}{4}  \int_{-\infty}^{\infty}  (1-W^4) \sech^2{r}\,dr\,.
    \end{equation}
    This shows that the energies $E_n:=E(W_n)$ are bounded from above by the energy $E_*:=E(W_*)=\ell(\ell+1)/2$ (see Fig.~\ref{fig2}).
    \begin{figure}[h]
    \includegraphics[width=0.495\textwidth]{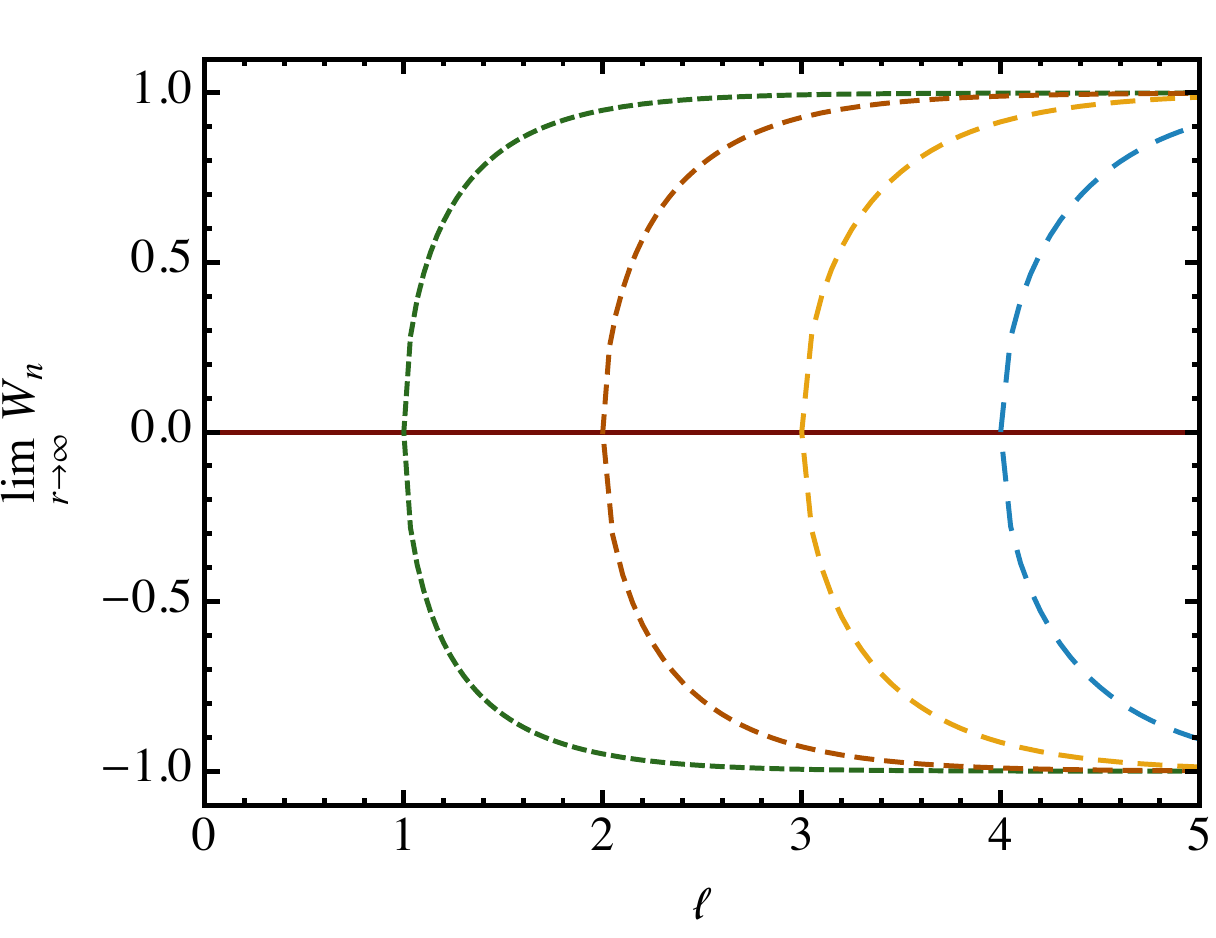}
\includegraphics[width=0.485\textwidth]{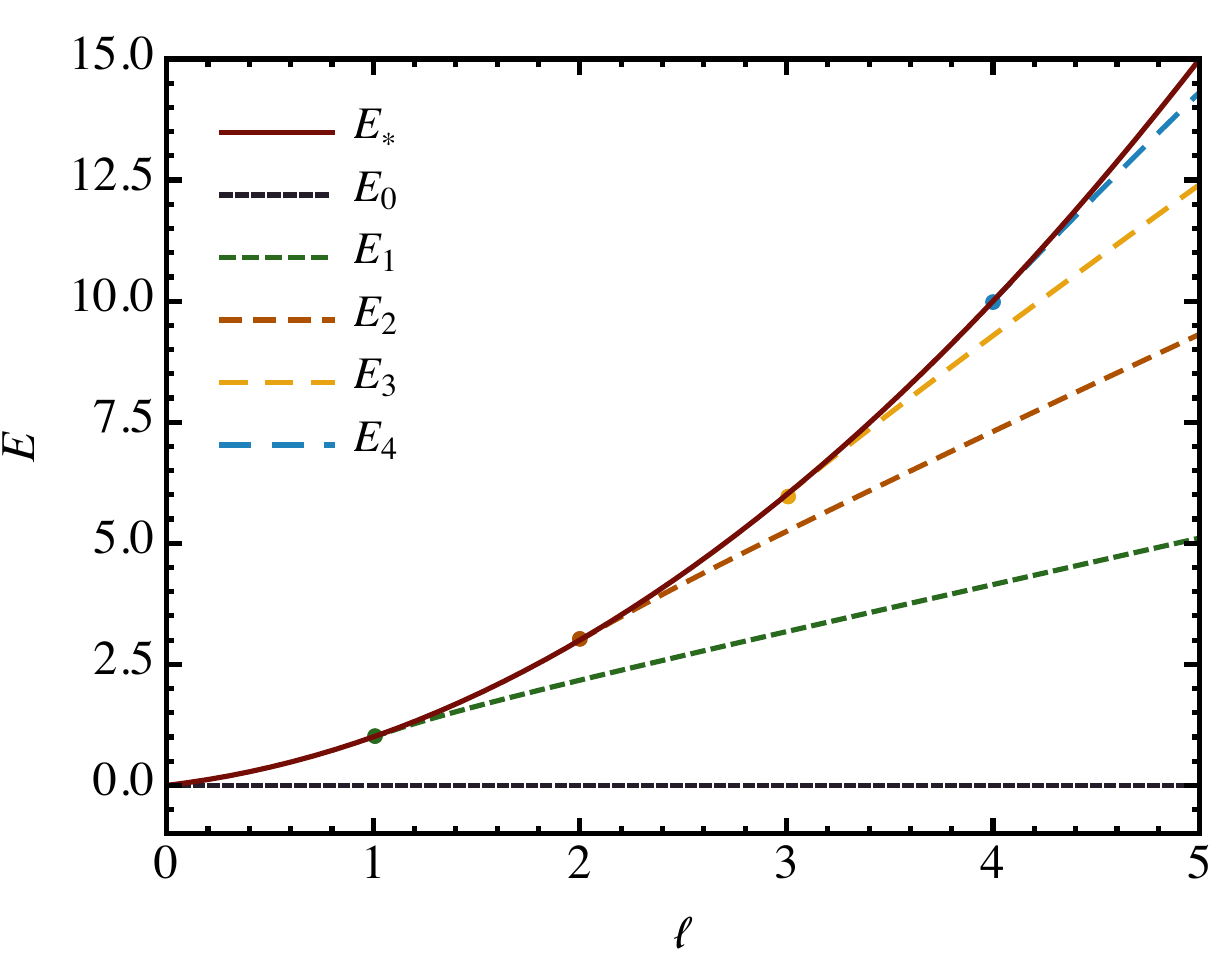}
\captionsetup{width=\textwidth}
\caption{\label{fig2} {\small Left panel: The bifurcation diagram showing a sequence of supercritical pitchfork bifurcations at integer values of $\ell$. Right panel: The energies $E_n$ are shown to bifurcate from  the parabola $E_*=\ell(\ell+1)/2$ at integer values of $\ell$. The bifurcation points are marked with dots.}}
\end{figure}
\end{rem}
\begin{rem}
Multiplying Eq.~\eqref{eqs} by $W'$ and integrating by parts we get an  identity
    \begin{equation}\label{ident}
      \int_{-\infty}^{\infty} \tanh{r} \sech^2{r} \,(1-W^2)^2\, dr= 0\,,
    \end{equation}
    which is trivially satisfied by odd and even solutions and suggests that non-symmetric solutions do not exist, however proving that is an open problem.
\end{rem}
\begin{rem}
For large  values of $\ell$ one can obtain an analytic approximation of the solution $W_1(r)$ as follows. Let $y=\sqrt{\ell(\ell+1)} r$ and $\bar W(y)=W(r)$. Then, in the limit $\ell\rightarrow \infty$ equation \eqref{eqs} reduces to
$\bar W''+\bar W (1-\bar W^2)=0$. The separatrix solution of this limiting equation  $\bar W_1(y)=\tanh(y/\sqrt{2})$ gives the   approximation
\begin{equation}\label{w1}
  W_1(r) \simeq \tanh\left(\sqrt{\frac{\ell(\ell+1)}{2}} r\right) \quad\mbox{for}\quad \ell\gg 1\,.
\end{equation}
\end{rem}
\begin{rem}
The  static energy
\begin{equation}\label{energy-s}
  E=\frac{1}{2} \int_{-\infty}^{\infty} \left( W'^2 + \frac{\ell(\ell+1)}{2\cosh^2{r}}\, (1-W^2)^2\right) \, dr\,
\end{equation}
is an example of an energy functional with the following properties: a) it is invariant under a discrete $Z_2$
 symmetry (here, the reflection $W\rightarrow -W$), b) the fixed point of this symmetry (here, $W_*$) is a critical point with maximal energy, c) it satisfies the Palais-Smale condition.
Corlette and Wald conjectured\footnote{Actually, this conjecture was not stated  explicitly in \cite{cw} but it follows naturally  from the argument given there.} in \cite{cw}, using Morse theory arguments,  that for such functionals the number of  critical points (counted without multiplicity) with energy below $E(W_*)$ is equal to the Morse index of $W_*$ (i.e. the number of negative eigenvalues of the Hessian of $E$ at $W_*$).
 We will show below that when $\ell<1$ the Morse index of $W_*$ is equal to $1$  and, as $\ell$ grows, it increases by one at each integer value of $\ell$.
 Thus, according to the conjecture of Corlette and Wald, for a given $\ell$ there should be exactly $n$ (besides $W_*$) critical points of the energy functional, where $n$ is the largest integer less than $\ell$. This is in perfect agreement with Theorem~1, provided that there are no non-symmetric critical points.
\end{rem}
\section{Linearized perturbations}
In this section we determine the linear stability properties of static solutions $W_n(r)$. This step is essential for understanding the role of static solutions in the evolution.
Following the standard procedure we seek solutions of Eq.~\eqref{eqym} in the form $W(t,r)=W_n(r)+w(t,r)$.
Hereafter, in the case of $W_*$ the subscript $n$ should be replaced by $*$.
Neglecting quadratic and cubic terms in $w$, we obtain the evolution equation for linear perturbations
  \begin{equation}\label{eq-linpert}
    \partial_{tt} w-\partial_{rr} w+V_n(r) \, w=0\,, \quad V_n(r)=\frac{\ell(\ell+1)}{\cosh^2{r}}\,(3 W_n^2(r)-1)\,.
\end{equation}
Separation of time dependence $w(t,r)=e^{\lambda t} v(r)$ yields  the eigenvalue problem for the one-dimensional Schr\"odinger operator
\begin{equation}\label{pert}
 L_n\, v:= \left(-\partial_{rr} +V_n(r)\right)\,v=\sigma\,v,\quad \sigma=-\lambda^2\,.
\end{equation}
  Since the potential $V_n(r)$ is an even function of~$r$, the eigenfunctions are alternately even and odd.
 We claim that the operator $L_n$ has exactly $n$ negative eigenvalues, independently of $\ell$. For $n=0$ this is obvious because the potential $V_0(r)=2\ell(\ell+1) \sech^{2}{r}$ is everywhere positive.  For  $n\geq 1$ we can obtain a lower bound as follows. Consider the function $v_n(r)=\cosh{r}\,W'_n(r)$. Multiplying Eq.~\eqref{eqs} by $\cosh^2{r}$ and then differentiating, we find that
   $ L_n\, v_n=-v_n$;
 hence $v_n$ is the eigenfunction to the eigenvalue $\sigma=-1$.
Since $W_n(r)$ has $(n-1)$ local extrema, the eigenfunction $v_n(r)$ has $(n-1)$ zeros which implies by
 the Sturm oscillation theorem  that there are exactly $(n-1)$ eigenvalues below~$-1$.  Consequently, the operator $L_n$  has at least $n$ negative eigenvalues. Numerics shows that there are no eigenvalues  in the interval $-1<\sigma<0$, indicating that the above lower bound is sharp but we have not been able to prove that (see, however, Remark~5 below).

  In the case of $W_*=0$ the potential $V_*=-\ell(\ell+1) \sech^2{r}$ is the exactly solvable P\"oschl-Teller potential \cite{dw} which is known to have $n+1$ negative eigenvalues $\sigma_j=-(\ell-j)^2$ for $j=0,1,\dots,n$, where $n$ is the largest integer less than $\ell$. In particular, for $\ell<1$ there is only one negative eigenvalue. At each positive integer value of $\ell$ a new zero energy resonance emerges from the bottom of the continuous spectrum and becomes a negative eigenvalue as $\ell$ grows.
\begin{rem}
From the structure of the spectrum of perturbations around $W_*$ it follows, using bifurcation theory, that at each positive integer value of $\ell$ there is a supercritical pitchfork bifurcation at which a pair of solutions $\pm W_n$ with $n=\ell$ is born. This gives an alternative ``soft'' argument for existence of solutions $W_n$ in some small intervals $n<\ell<n+\varepsilon$. In these intervals the solutions $W_n$ are guaranteed to have one unstable mode less than $W_*$, in agreement with  our conjecture above.
\end{rem}

  In order to understand the evolution in the neighbourhood of  static solutions, besides unstable modes,  one needs to determine the quasinormal modes. They are defined as solutions of the eigenvalue equation \eqref{pert} with $\mathrm{Re}(\lambda)<0$ and the outgoing wave conditions  $v(r)\sim  \exp(\mp \lambda r)$ for $r\rightarrow  \pm \infty$. As the concept of quasinormal modes is inherently related to the loss of energy by radiation, the unitary evolution \eqref{eq-linpert} and the associated  self-adjoint eigenvalue problem  \eqref{pert} do not provide a natural setting for analysing quasinormal modes, both from the conceptual and computational viewpoints. For this reason we postpone the discussion of quasinormal modes until the next section where a new non-unitary formulation will be introduced.
\section{Hyperboloidal formulation}
We define  new  coordinates
\begin{equation}\label{tau-x}
 \tau=t-\log(\cosh{r})\,,\quad x=\tanh{r}\,.
 \end{equation}
 Then metric \eqref{g} takes the form
\begin{equation}\label{gh}
  g=(1-x^2)^{-1}\, \hat g, \quad \hat g=-(1-x^2) d\tau^2-2x \, d\tau dx + dx^2 + \alpha^2\,d\omega^2\,.
\end{equation}
 The conformal metric $\hat g$ is  diffeomorphic to the Nariai metric \cite{n}, which is a product of the 2-dimensional de Sitter metric   and the round metric on the 2-sphere. It has constant positive scalar curvature $R(\hat g)=2+2/\alpha^2$. The manifold $\mathcal{M}$ is the static patch of dS$_2 \times S^2$ whose cosmological horizons (null hypersurfaces $x=\pm 1$) correspond to null infinities of $\mathcal{M}$.
 In $(\mathcal{M},\hat g)$ the spacelike hypersurfaces $\tau=$ const are three-dimensional cylinders $R\times S^2$.  Note that the Yang-Mills equation is conformally invariant in four dimensions, hence it takes the same form
on $(\mathcal{M},g)$ and $(\mathcal{M},\hat g)$.
 The hypersurfaces  $\tau=$ const are ``hyperboloidal'', that is they are spacelike hypersurfaces that approach the ``right'' future null infinity  ${\mathcal J}_R^+$ along  outgoing null cones of constant retarded time $t-r$ and the ``left'' future null infinity  ${\mathcal J}_L^+$ along outgoing null cones of constant advanced time $t+r$.

We recall that the hyperboloidal approach to the initial value problem was introduced by Friedrich  in his studies of asymptotically flat solutions of Einstein's equations \cite{f}. More recently, this approach has been developed and  applied  by Zengino\u{g}lu \cite{z1,z2}, who emphasized its advantages over the traditional approaches (see also \cite{mr}).
Until now, as far as we know, the hyperboloidal approach has not been used in the context of asymptotically hyperbolic spacetimes.

In terms of the coordinates $(\tau,x)$ Eq.~\eqref{eqym} becomes\footnote{Hereafter we  abuse notation and use the same letter  $W$ for the dependent variable  in coordinates $(t,r)$ and $(\tau,x)$.}
\begin{equation}\label{eq-sx}
  \partial_{\tau\tau} W+2x \partial_{\tau x} W + \partial_{\tau} W = \partial_x\left((1-x^2) \partial_x W\right)+ \ell(\ell+1) W(1-W^2)\,.
\end{equation}
The principal part of this hyperbolic equation degenerates at the endpoints $x=\pm 1$ to $\partial_{\tau} (\partial_{\tau}\pm 2\partial_x) W$, hence there are no ingoing characteristics at the boundaries and consequently  no boundary conditions are required or, for that matter,  allowed. This, of course, reflects the fact that no information comes in from the future null infinities.

Multiplying Eq.~\eqref{eq-sx} by $\partial_{\tau} W$ we obtain the local conservation law
\begin{equation}\label{cons}
 \partial_{\tau} \rho +\partial_x f=0\,,
\end{equation}
where
\begin{eqnarray}\label{cons2}
\rho&=& \frac{1}{2} \left((\partial_{\tau} W)^2 + (1-x^2)(\partial_x W)^2
  +\frac{\ell(\ell+1)}{2}\,(1-W^2)^2\right) \,,\\
f&=& x\,(\partial_{\tau} W)^2 - (1-x^2) \partial_{\tau} W \partial_x W\,.
\end{eqnarray}
Integrating the conservation law \eqref{cons} over a $\tau=$ const hypersurface  we get the energy balance
\begin{equation}\label{de}
  \frac{d \mathcal{E}}{d\tau} = -(\partial_{\tau} W(\tau,-1))^2-(\partial_{\tau} W(\tau,1))^2\,,
\end{equation}
where
\begin{equation}\label{energy}
  \mathcal{E}(\tau)= \int_{-1}^{1} \rho \,dx
\end{equation}
is the Bondi-type energy\footnote{Note that
$\partial_{\tau}=\partial_t$, hence for time-independent fields the Bondi-type energy $\mathcal{E}$ is equal to the standard energy $E$ defined in \eqref{energy-r}.}.
Formula \eqref{de} expresses the radiative loss of energy  through future null infinities. Since the energy $\mathcal{E}(\tau)$ is positive and monotone decreasing, it has a nonnegative limit for $\tau\rightarrow \infty$. It is natural to expect
that this limit is given by the energy of a static endstate of evolution\footnote{In order to turn this expectation into a proof, it would suffice to  show that the `kinetic
energy' $\frac{1}{2}\,\int_{-1}^{1} (\partial_{\tau} W)^2 \,dx$ must tend to zero as $\tau \rightarrow \infty$.
}.
 We thus see that the hyperboloidal approach is ideally suited for studying the relaxation processes which are governed by  the dispersive dissipation of energy.

 In the remainder of the paper we describe in detail the asymptotic behavior of solutions to Eq.~\eqref{eq-sx} for smooth, finite energy initial data. To this end, we first
return to the analysis of linearized perturbations and redo it in the hyperboloidal framework.
Substituting $W(\tau,x)=W_n(x)+w(\tau,x)$ into Eq.~\eqref{eq-sx} and linearizing, we obtain the equation
\begin{equation}\label{eq-sxl}
  \partial_{\tau\tau} w+2x \partial_{\tau x} w + \partial_{\tau} w = \partial_x\left((1-x^2) \partial_x w\right)+ \ell(\ell+1) (1-3W_n^2)\,w\,,
\end{equation}
 which for $w(\tau,x)=e^{\lambda \tau} u(x)$ yields the eigenvalue problem for the quadratic pencil of linear operators
\begin{equation}\label{pert-sx}
   (A_n   +\lambda\,B +\lambda^2 I)\, u =0\,,
\end{equation}
where
$$
A_n=-\partial_x \left((1-x^2) \partial_x\right) + \ell(\ell+1) (3 W_n^2(x)-1),\quad B=2x \partial_x +1\,.
$$
Here $A_n$ is the self-adjoint operator (corresponding to the operator $L_n$ defined in \eqref{pert}), while the operator $B$ is skew-symmetric. Thus, if $\lambda$ is an eigenvalue\footnote{Mind the  change of terminology: in the previous section  $\sigma=-\lambda^2$ was referred to as the eigenvalue while for the quadratic pencil $\lambda$ is called an eigenvalue.}, so is its complex conjugate $\bar \lambda$.
In this setting the quantization condition for the  eigenvalues amounts to the requirement that the eigenfunctions be  smooth at the endpoints $x=\pm 1$. Note that here, in contrast to the self-adjoint formulation from section~3, both unstable and stable (quasinormal) modes are treated on equal footing as genuine eigenfunctions\footnote{To our knowledge, the advantages of the  hyperboloidal foliations in the definition and analysis of   quasinormal modes were first pointed out by  Schmidt \cite{bernd}. More recently, this idea was implemented numerically in \cite{brz} and, in the case of asymptotically anti-de Sitter black holes, was independently developed rigorously by Warnick  in the framework of semigroup theory \cite{warnick}.}.

 For constant solutions, $W_*=0$ and $W_0=1$, one can solve the quadratic eigenvalue problem \eqref{pert-sx} using the power series method. In the case of $W_*$, inserting the power series
 \begin{equation}\label{power}
 u(x)=\sum_{j=0}^{\infty} c_j x^j
 \end{equation}
  into \eqref{pert-sx} we obtain the recurrence relation
\begin{equation}\label{recurrence*}
  c_{j+2}=\frac{(j+\lambda+\ell+1)(j+\lambda-\ell)}{(j+2)(j+1)}\, c_j\,,
\end{equation}
where $c_0=1$, $c_1=0$ for even solutions and $c_0=0$, $c_1=1$ for odd solutions.
Using the ratio test and the Gauss convergence criterion it is easy to see that the function defined by the power series \eqref{power} is  smooth for  $|x|< 1$, but it is not smooth at $x= \pm 1$ unless the series is truncated for a finite $j$. The truncation condition gives two sequences of eigenvalues
\begin{equation}\label{eigen-w*}
  \lambda_{*,j}^{\pm}=-j-\frac{1}{2}\pm (\ell+\frac{1}{2}),\qquad j=0,1,\dots
\end{equation}
The corresponding eigenfunctions $u_{n,j}^{\pm}$ are even and odd polynomials of order $j$ for even and odd $j$, respectively.
It follows from \eqref{eigen-w*} that $W_*$ has one (even) unstable mode $(\lambda_{*,0}^+=\ell,u_{*,0}^+=1)$, if $\ell<1$, and as $\ell$ increases, it picks up new unstable modes at each integer value of $\ell$.  Thus, for a given $\ell$ there are exactly $n+1$ unstable modes where $n$ is the largest
integer less than $\ell$.

In the case of $W_0$  the recurrence relation is
\begin{equation}\label{recurrence0}
  c_{j+2}=\frac{\lambda^2+(2+j)\lambda+j^2+j+2\ell(\ell+1)}{(j+2)(j+1)}\, c_j\,,
\end{equation}
and the truncation condition yields two sequences of eigenvalues
\begin{equation}\label{eigen-w0}
  \lambda_{0,j}^{\pm} =-j-\frac{1}{2}\pm \frac{1}{2} \sqrt{1-8\ell(\ell+1)},\qquad  j=0,1,\dots
\end{equation}
All the eigenvalues have a negative real part (in agreement with the analysis in section~3). An imaginary part is nonzero if  $\ell>(\sqrt{6}-2)/4$.

The power series method is not applicable to the  solutions $W_n$ with $n\geq 1$ because they are not known in closed form, hence in order to compute their spectrum of perturbations one has to resort to numerical  methods, for example the shooting method\footnote{In the case of $W_1$, the single unstable mode is even so  the least damped stable mode can be easily obtained by solving the linearized equation \eqref{eq-sxl} for some odd initial data and fitting an exponentially damped sinusoid to $w(\tau,x_0)$ at some~$x_0$.} (see \cite{bcr2}). In this case we shall adopt the following convention for the ordering of  eigenvalues
$$
\lambda_{n,0}> \dots > \lambda_{n,n-1}=1>0>\mathrm{Re}(\lambda_{n,n})>\mathrm{Re}(\lambda_{n,n+1})>\dots
$$

Summarizing, the following picture emerges from our analysis.  For a given $\ell$ there are $n+2$ static solutions $W_0, W_1,\dots,W_n, W_*$, where $n$ is the largest integer less that $\ell$. Due to the global-in-time regularity of solutions and monotone decrease of the energy $\mathcal{E}$, the static solutions are expected to be the only possible endstates of evolution for any  smooth, finite energy initial data.  Since the solution $W_n$ has $n$ unstable modes (or $n+1$ in the case of $W_*$), only $W_0$ is a generic attractor while the solutions $W_n$ with $n\geq 1$ are unstable attractors of codimension $n$ (or $n+1$ in the case of $W_*$).
 The analytic and numerical evidence supporting this picture will be given in the following section where we describe in detail  the dynamics  of convergence to the static attractors.
 \section{Relaxation to an equilibrium}
 In this section we solve Eq.~\eqref{eq-sx} numerically  for a variety of smooth, finite energy initial data and different values of $\ell$. We use the spectral Galerkin method which appears to be ideally suited to the problem at hand. The implementation of this method is presented in detail in  Appendix A. Our goal is to describe  quantitatively the process of relaxation to static solutions. We will discuss in turn the relaxation to the generic attractor $W_0$, the codimension-one attractors $W_*$ (for $\ell\leq 1$) or $W_1$ (for $\ell>1$), the codimension-two attractors $W_*$ (for $1<\ell<2$) or $W_2$ (for $\ell>2$), and the codimension-three attractors $W_*$ (for $2<\ell<3$) or $W_3$ (for $\ell>3$).

 \subsection{Relaxation to $W_0$}
 Generic initial data are expected to evolve towards the static solution $W_0$; our numerical results confirm this expectation. Fig.~\ref{fig3} shows the evolution for $\ell = 7/2$ of sample initial data
 \begin{equation}\label{ic0}
 W(0,x) = T_1(x) + T_4(x) + T_5(x), \quad \partial_{\tau} W(0,x) = 0\,.
 \end{equation}
  Here and in the following, $T_n(x)$ denotes the $n$-th Chebyshev polynomial. We choose initial data as a simple combination of Chebyshev polynomials because they can be easily implemented in the framework of our numerical method. Note that the first two eigenfunctions around $W_0$ and $W_*$ are given by the first two Chebyshev polynomials: $u_{0,0}^{\pm}=u_{*,0}^{\pm}=T_0\equiv 1$ and $u_{0,1}^{\pm}=u_{*,1}^{\pm}=T_1\equiv x$. This fact will be very convenient in the following, however it should be remembered that it is specific to our choice of the hyperboloidal coordinates \eqref{tau-x}.
\begin{figure}
\includegraphics[width= \textwidth]{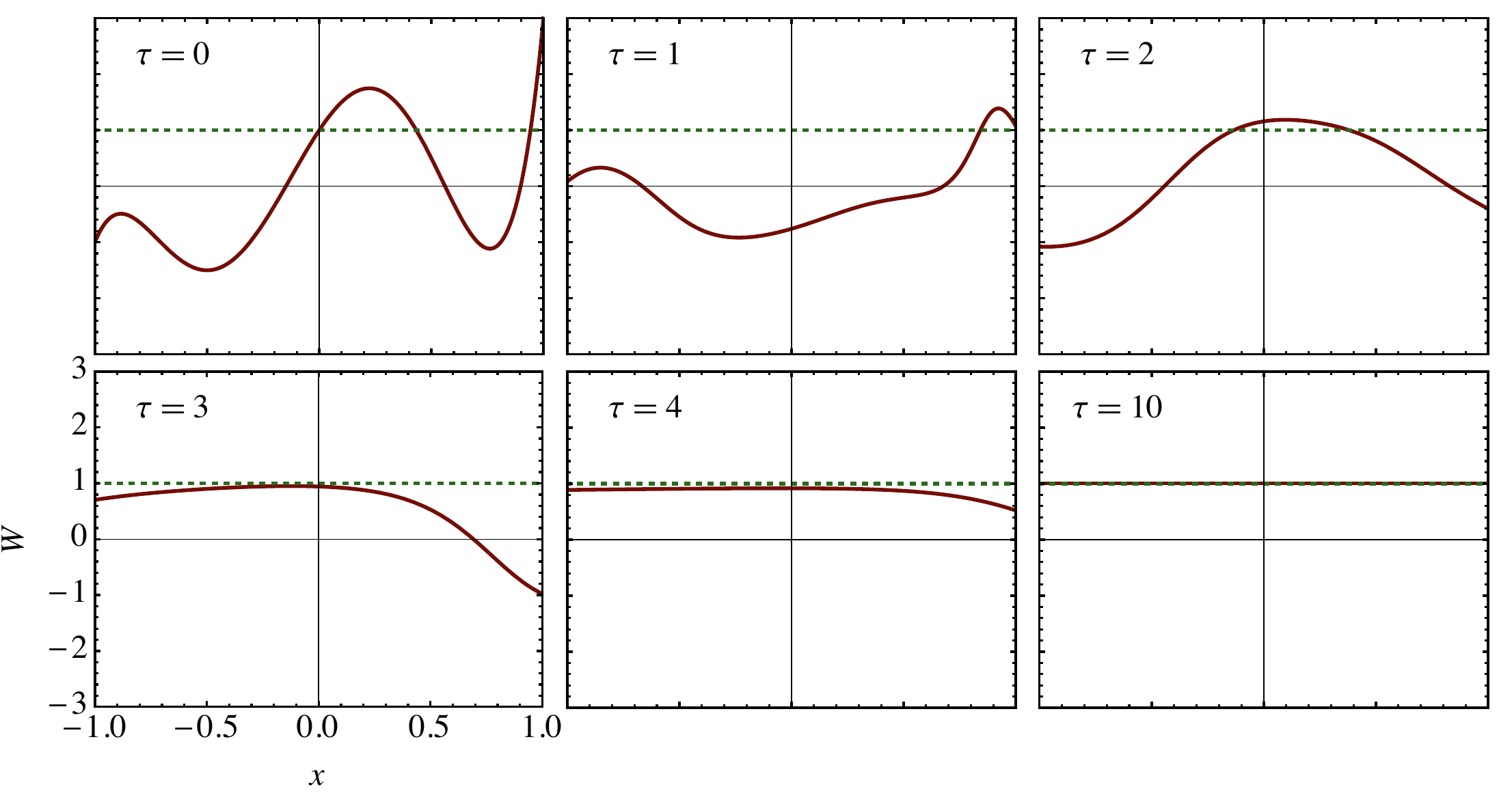}
\captionsetup{width=\textwidth}
\caption{\label{fig3} {\small Relaxation to $W_0$ for $\ell = 7/2$ and the initial data \eqref{ic0}.}}
\end{figure}

We find that the convergence to $W_0$ is determined by the least damped mode with the eigenvalue and eigenfunction
\begin{equation}\label{mode00}
\lambda_{0,0}^{+} =-\frac{1}{2}+ \frac{1}{2} \sqrt{1-8\ell(\ell+1)}\,,\quad u_{0,0}^{+}(x)=1\,,
\end{equation}
that is, for $\tau\rightarrow \infty$\footnote{The notation  $f(\tau) \sim g(\tau)$ for $\tau \rightarrow \infty$ means that $f(\tau)$ is ``asymptotically equivalent'' to $g(\tau)$, i.e. $\lim_{\tau \rightarrow \infty} f(\tau)/g(\tau)=1$.}
\begin{equation}\label{conW0}
   W(\tau,x)-W_0 \sim \mathrm{Re}\left(A_0 e^{\lambda_{0,0}^+ \tau}\right) + \mbox{subleading terms}\,,
\end{equation}
where $A$ is a constant depending on the initial data. This asymptotics can be formally derived using the Galerkin approximation of the infinite dimensional dynamical system \eqref{coeffeq}. To see this, note that the solution $W_0\equiv 1$ corresponds to the fixed point $a_k=\delta_k^0$ of the system \eqref{coeffeq}. Truncating this system at $N=2$ we obtain
\begin{subequations}
\begin{eqnarray}
 \ddot a_0 + \dot a_0 - \ell (\ell + 1)a_0 + \ell (\ell + 1) \left(a_0^3  + \frac{3}{2} a_0 a_1^2  \right) =  0,&& \\
 \ddot a_1 + 3 \dot a_1 + \left(2 - \ell (\ell + 1)\right) a_1 + \ell (\ell + 1) \left(3 a_0^2 a_1+\frac{3}{4} a_1^3\right) =  0.&&
\end{eqnarray}
\end{subequations}
For this system it is routine to show that the convergence to the fixed point  ($a_0=1,\dot a_0=0,a_1=0,\dot a_1=0$) for $\tau\rightarrow \infty$ has the form
\begin{subequations}
\begin{eqnarray}\label{fixed0}
  a_0(\tau) &\sim & 1+ \mathrm{Re}\left(A_0 e^{(-\frac{1}{2}+\frac{1}{2} \sqrt{1-8\ell(\ell+1)}) \tau}\right),\\
  a_1(\tau) &\sim & \mathrm{Re}\left(A_1 e^{(-\frac{3}{2}+\frac{1}{2} \sqrt{1-8\ell(\ell+1)})\tau}\right)\,.
\end{eqnarray}
\end{subequations}
Increasing the order $N$ of the Galerkin approximation does not affect the leading order asymptotics of $a_0(\tau)$ and $a_1(\tau)$ and the increasingly faster fall-off of higher coefficients can be calculated systematically term by term, therefore \eqref{conW0} follows.
The numerical verification of the  asymptotic behavior (40) is shown in Fig.~\ref{fig4}.

\begin{figure}[h!]
\includegraphics[width=0.49\textwidth]{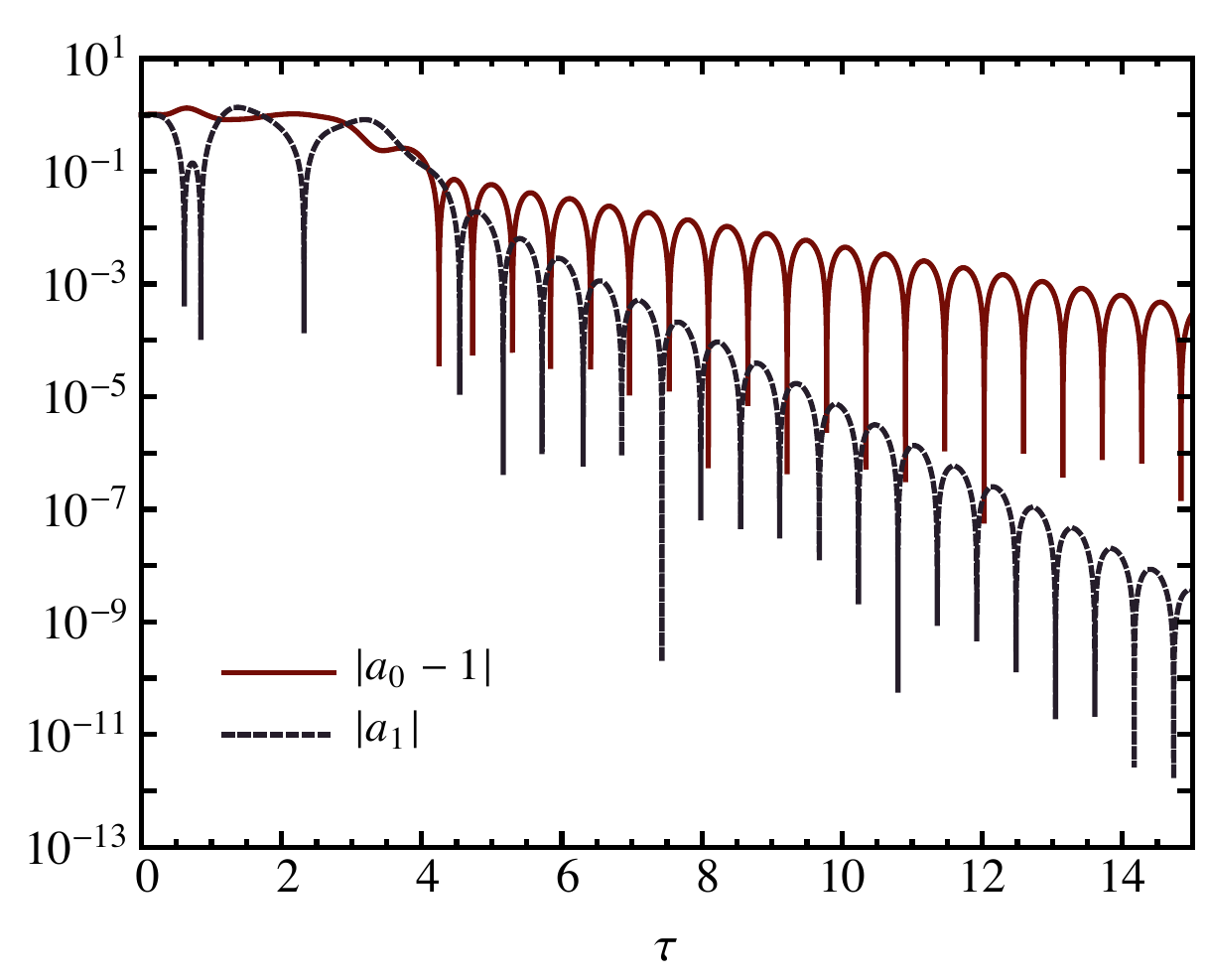}
\includegraphics[width=0.49\textwidth]{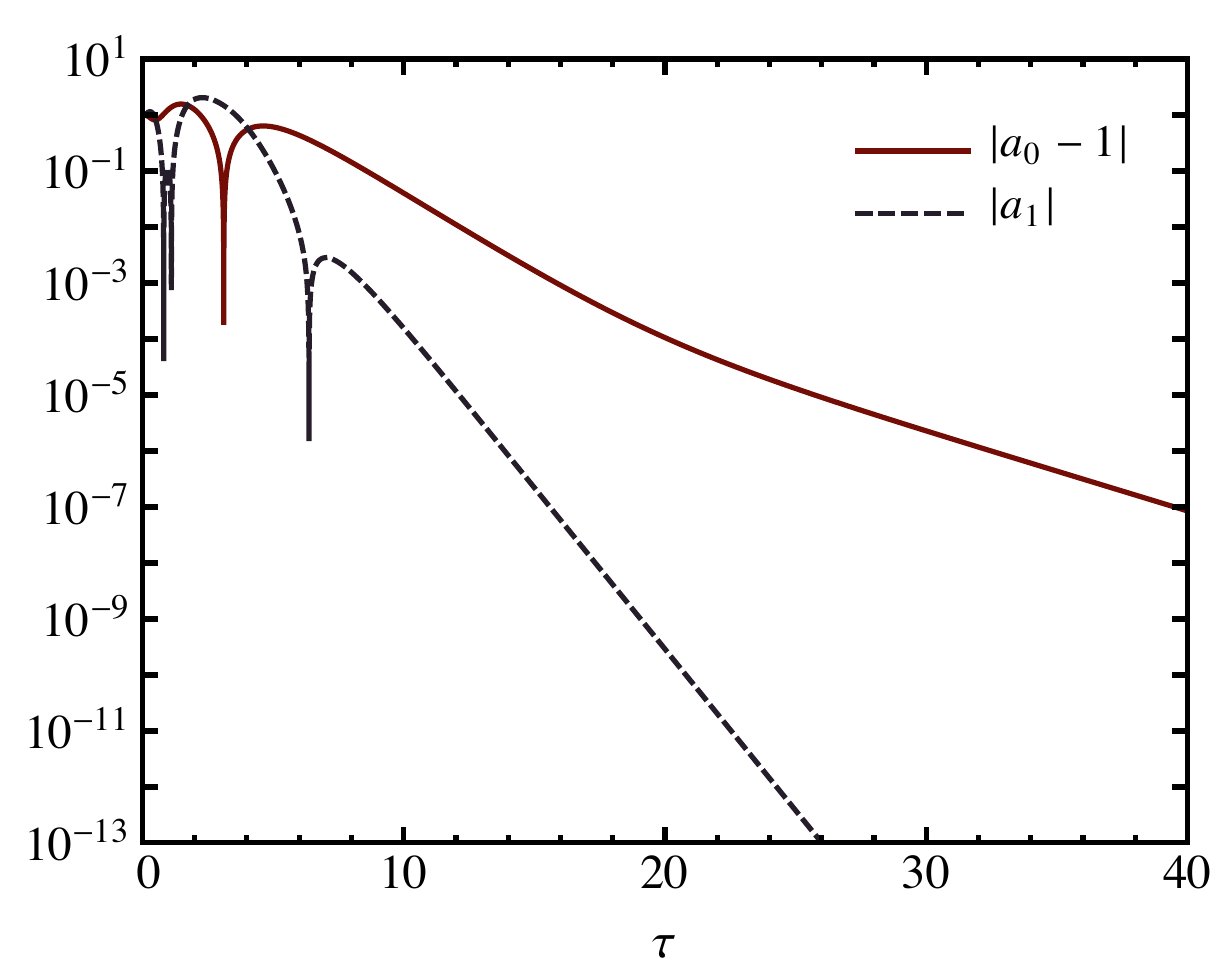}
\captionsetup{width=\textwidth}
\caption{\label{fig4} {\small Illustration of the asymptotic behavior (40). The plots depict
the quantities $|a_0(\tau) - 1|$ and $|a_1(\tau)|$ for the solution with initial data \eqref{ic0}. In the oscillatory case $\ell=7/2$ (left panel) both quantities
oscillate with the frequency $5\sqrt{5}/2$ and fall off as $e^{-\tau/2}$ and $e^{-3 \tau/2}$, respectively.
In the non-oscillatory case $\ell = 1/10$ (right panel) the quantities fall off as $e^{(-5 + \sqrt{3}) \tau /10}$ and $e^{(-15 + \sqrt{3}) \tau /10}$, respectively. Note that the asymptotic falloff  of $|a_0(\tau) - 1|$ becomes visible only for $\tau \gtrsim 25$.}}
\end{figure}

\subsection{Codimension-one attractors}
 For odd initial data $W_0$ cannot be an attractor (because it is even) and the role of an  attractor is taken by $W_*$ (if $\ell\leq 1$) or $W_1$ (if $\ell>1$). Each of these solutions has one unstable mode, but this mode is even so it does not participate in the evolution of odd initial data. More precisely, odd initial data lie on the stable manifold of $W_*$ if  $\ell<1$, the center-stable manifold of $W_*$ if $\ell=1$, and the stable manifold of $W_1$ if $\ell>1$. Below we describe in turn the relaxation to these attractors.

 \subsubsection{Relaxation to $W_*$ for $\ell \leq 1$.}
  \begin{figure}[h]
\includegraphics[width=\textwidth]{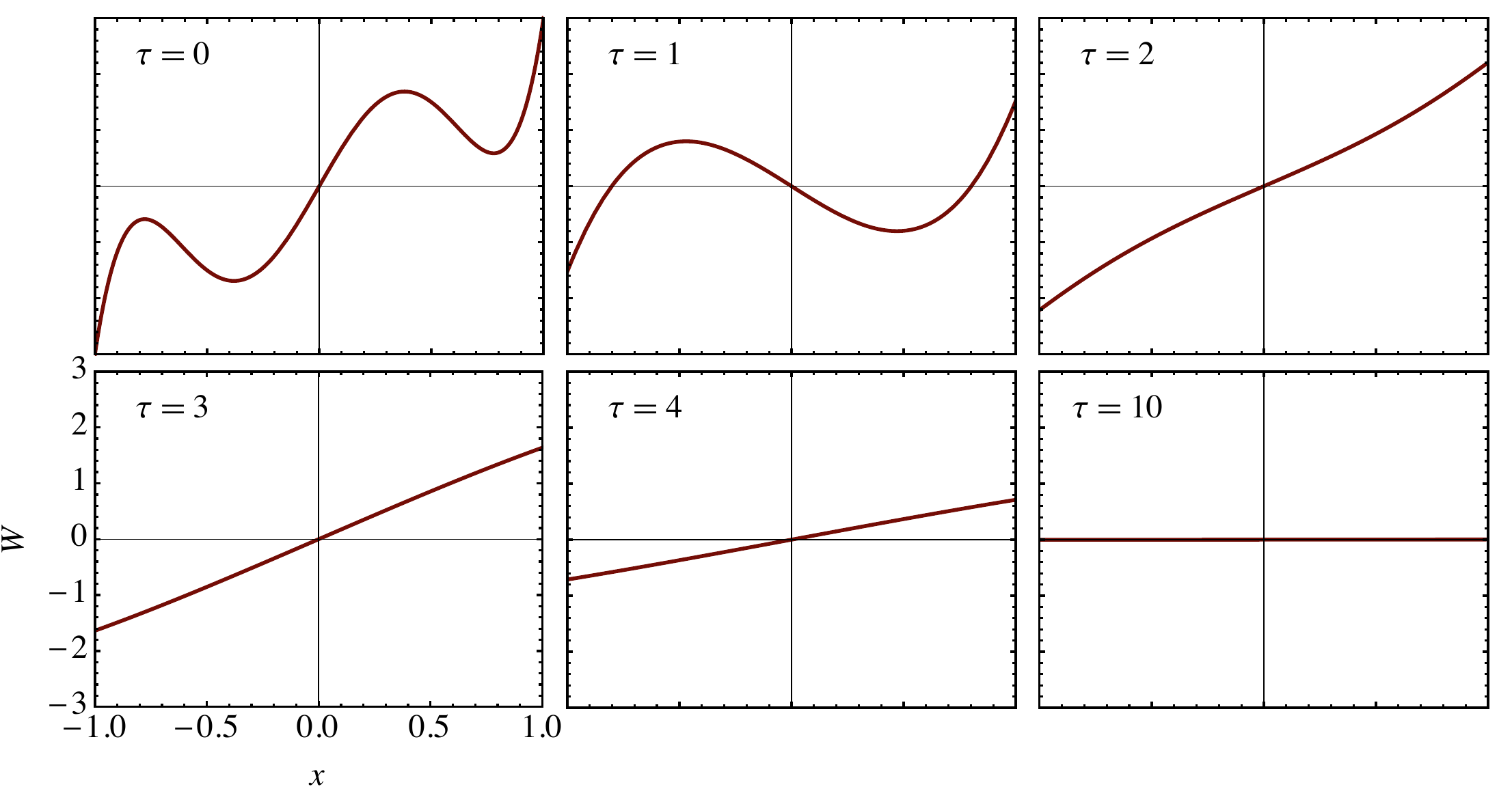}
\captionsetup{width=\textwidth}
\caption{\label{fig5} {\small Relaxation to $W_*$ for $\ell = 1/8$ and odd initial data \eqref{ic-odd}.}}
\end{figure}
 As follows from \eqref{eigen-w*} and \eqref{recurrence*}, the least damped mode has the  eigenvalue and eigenfunction
 \begin{equation}\label{mode0*}
\lambda_{*,1}^{+} =\ell-1\,,\quad u_{*,0}^{+}(x)=x\,.
\end{equation}
For $\ell<1$, in analogy to \eqref{conW0}, we have for $\tau\rightarrow \infty$
\begin{equation}\label{conW*}
   W(\tau,x)\sim A e^{\lambda_{*,1}^+ \tau}\,x + \mbox{subleading terms}\,.
\end{equation}
As above, this asymptotic behavior can be formally derived from the Galerkin approximation whose truncation to the first two odd modes reduces to
\begin{subequations}
\begin{align}
  \ddot a_1 + 3\dot a_1  + \left(2 - \ell (\ell + 1) \right) a_1 &+12 \dot a_3 + 6 a_3 \nonumber \\
+   &\frac{3}{4} \ell (\ell + 1) \left(a_1^3 + a_1^2 a_3  + 2 a_3^2 a_1 \right) =  0,
\end{align}
\begin{equation}
  \ddot a_3 +  7 \dot a_3  + \left(12 -\ell (\ell +  1) \right) a_3
+ \frac{1}{4} \ell (\ell  +  1) \left(a_1^3  +  6 a_1^2 a_3  +  3 a_3^3 \right)  =  0.
\end{equation}
\end{subequations}
Solving this system asymptotically  for $\tau\rightarrow \infty$ near the fixed point  $(a_1=\dot a_1=a_3=\dot a_3=0)$ we get
\begin{subequations}
\begin{eqnarray}\label{fixed*}
  a_1(\tau) &\sim &  A e^{(\ell-1) \tau},\\
  a_3(\tau) &\sim & -\frac{1}{8} \frac{1+\ell}{1+4\ell}\, A^3 e^{3(\ell-1) \tau}\,.
\end{eqnarray}
\end{subequations}
 Note that the leading fall-off of $a_3(\tau)$ is governed by the nonlinear term proportional to $a_1^3$ in Eq.~(43b).

For $\ell=1$ the eigenvalue $\lambda_{*,1}^{+}=\ell-1$ is equal to zero\footnote{In the framework of the self-adjoint problem \eqref{pert} this corresponds to the presence of the zero energy resonance at  the bottom of the continuous spectrum.}. In this case the coefficient of the term $a_1$ in Eq. (43a) vanishes and the decay changes from exponential to power-law, namely
\begin{subequations}
\begin{eqnarray}\label{fixed1}
  a_1(\tau) &\sim & \pm \frac{\sqrt{5}}{2}\,\tau^{-1/2} ,\\
  a_3(\tau) &\sim & \mp \frac{\sqrt{5}}{32}\,\tau^{-3/2}\,.
\end{eqnarray}
\end{subequations}
This nongeneric case is somewhat surprising as it contradicts a naive expectation that the decay is always exponential\footnote{A similar power-law decay occurs for other integer values of $\ell$ for specially prepared odd initial data.}.

Fig.~\ref{fig5} depicts the relaxation to $W_*$ for sample odd initial data
\begin{equation}\label{ic-odd}
   W(0,x) = 2 T_1(x) + T_5(x), \quad \partial_{\tau} W(0,x) = 0\,.
\end{equation}

Increasing the order of the Galerkin approximation does not affect the asymptotics (44) and (45).
The numerical verification of these formulae is shown in Fig.~\ref{fig6}.
\begin{figure} [h]
\includegraphics[width=0.48\textwidth]{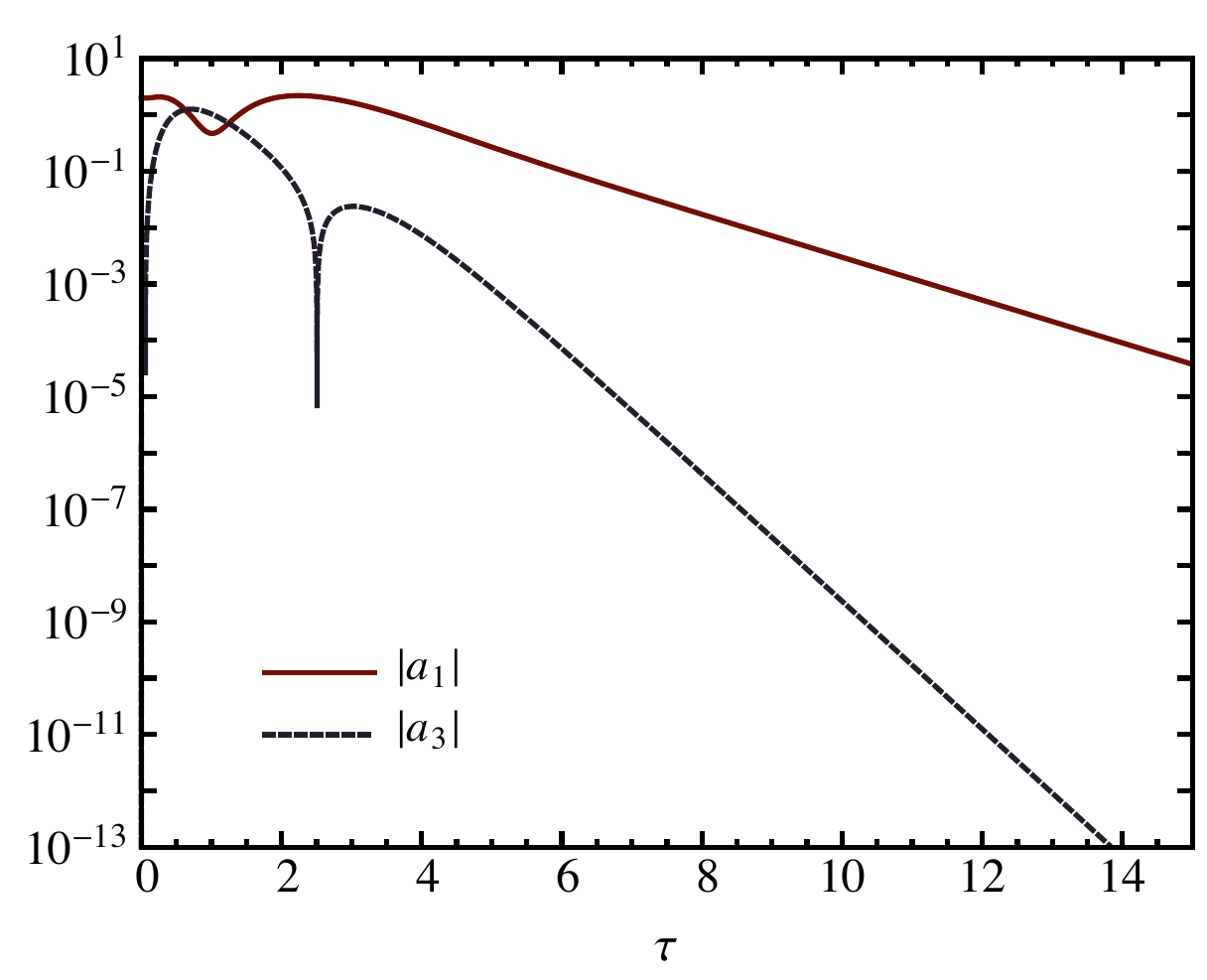}
\includegraphics[width=0.48\textwidth]{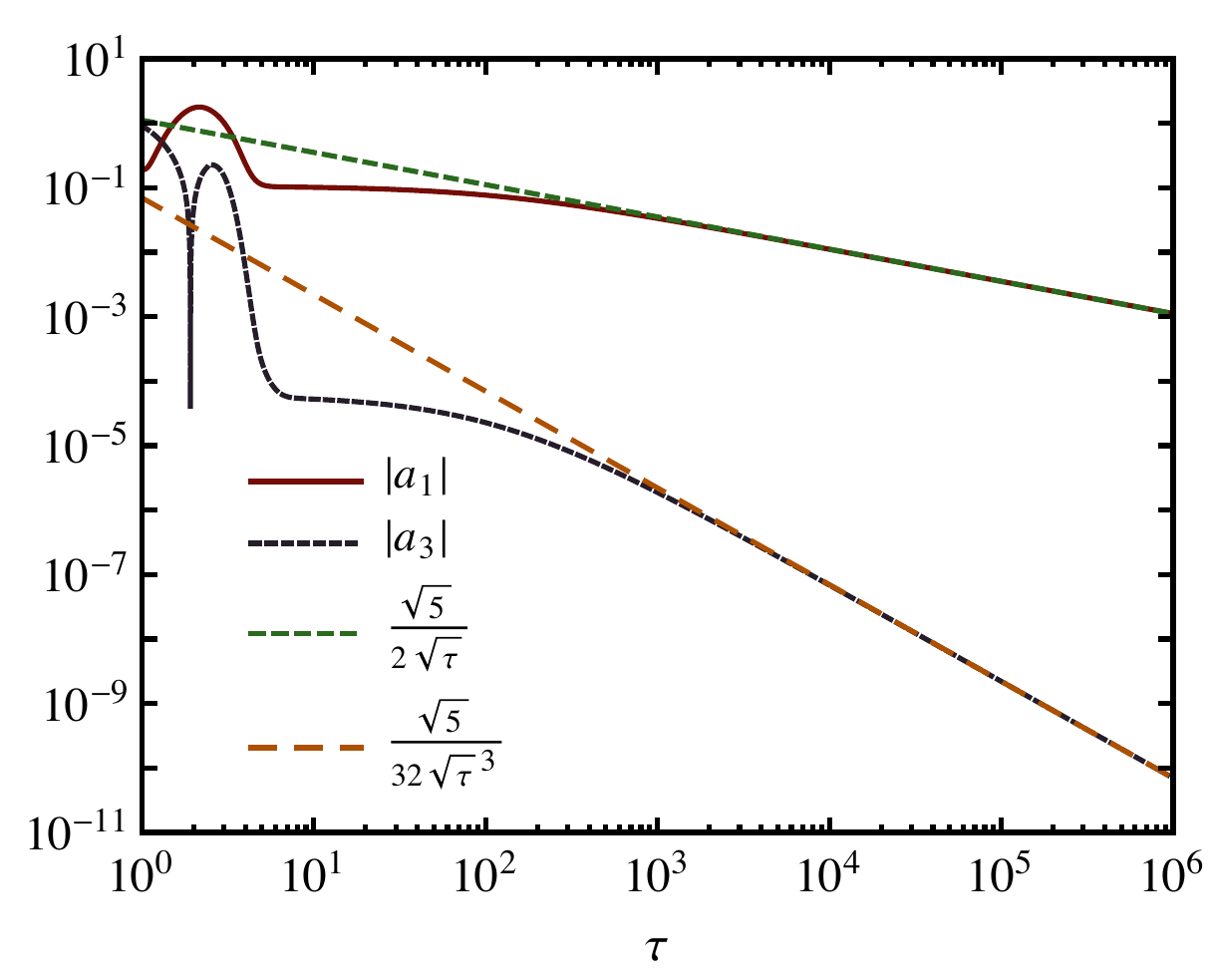}
\captionsetup{width=\textwidth}
\caption{\label{fig6} {\small Illustration of the asymptotics (44) and (45). The plots depict the coefficients $|a_1(\tau)|$ and  $|a_3(\tau)|$ for the solution with initial data \eqref{ic-odd}. The left panel shows the case $\ell = 1/8$. As expected, the coefficients $|a_1(\tau)|$ and $|a_3(\tau)|$ decay as $A e^{-7 \tau/8}$ and $(3 A^3/32) e^{-21 \tau/8}$, with $A \approx 18.8$. The right panel, using the double logarithmic scale, shows the nongeneric case $\ell = 1$. To make the predicted asymptotic power-law behavior of $|a_1(\tau)|$ and $|a_3(\tau)|$ better visible, we superimpose the graphs of $\sqrt{5}/(2 \sqrt{\tau})$ and $\sqrt{5}/(32 \sqrt{\tau}^3)$.}}
\end{figure}

 \subsubsection{Relaxation to $W_1$ for $\ell>1$.}
The single unstable mode around $W_1$ is even, hence it is not present in the evolution of odd initial data. Therefore, in this case we have
\begin{equation}\label{conW1}
   W(\tau,x)-W_1(x) \sim A e^{-\gamma_1 \tau}\, \sin(\omega_1 \tau+\delta)\,u_{1,1}(x)+...\,,
\end{equation}
where
$\gamma_1=-\mathrm{Re}(\lambda_{1,1})$, $\omega_1=\mathrm{Im}(\lambda_{1,1})$ and $A$, $\delta$ are constants.
Figs.~\ref{fig7} and \ref{fig8} show the numerical results confirming the asymptotic behavior  \eqref{conW1}  for $\ell=7/2$ and the initial data \eqref{ic-odd}.

\begin{figure} [h]
\includegraphics[width=\textwidth]{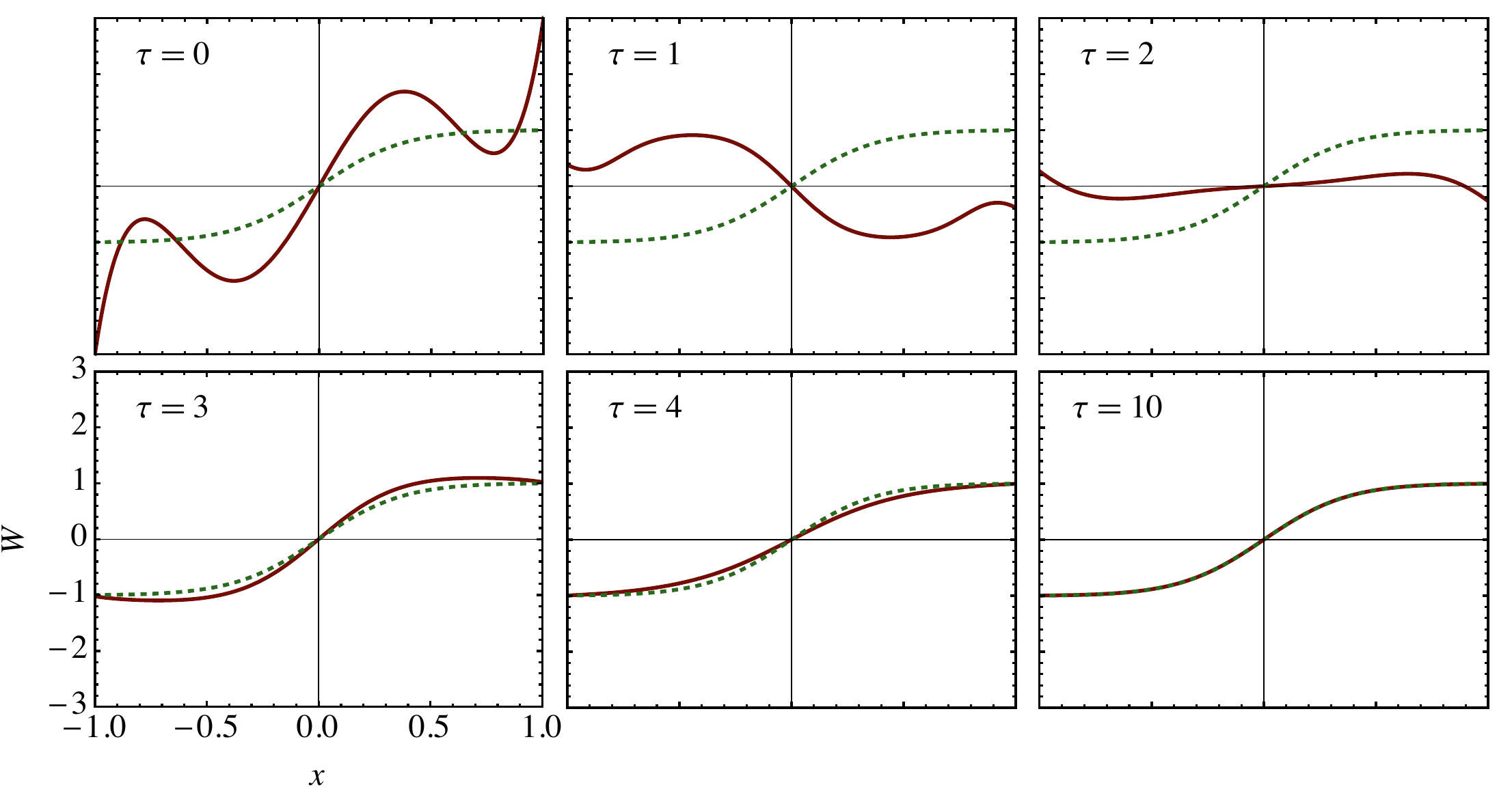}
\captionsetup{width=\textwidth}
\caption{\label{fig7} {\small Relaxation to $W_1$ (depicted with the dashed line) for $\ell = 7/2$ and the initial data \eqref{ic-odd}.}}
\end{figure}

\begin{figure} [h]
\includegraphics[width=0.5 \textwidth]{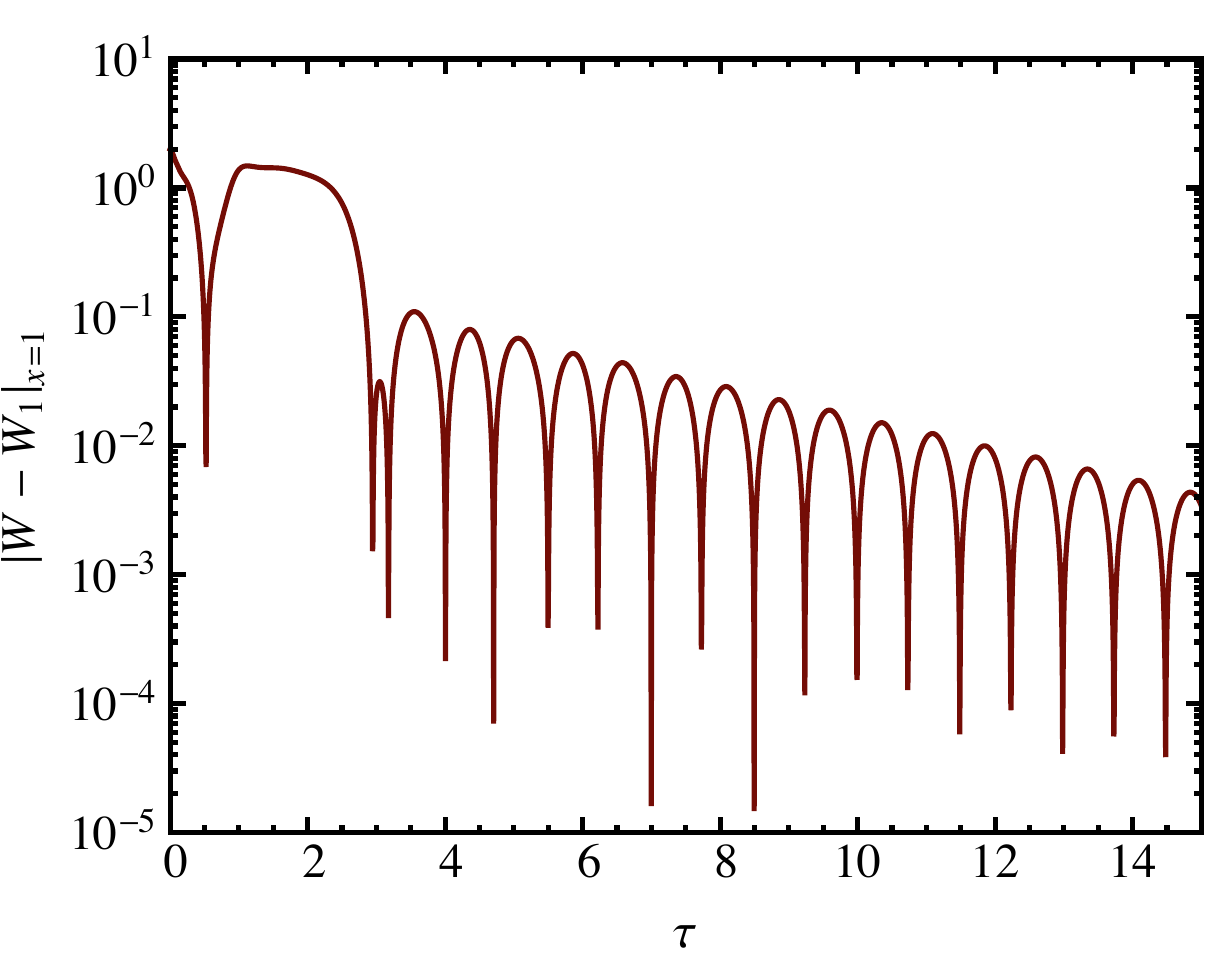}
\captionsetup{width=\textwidth}
\caption{\label{fig8} {\small Time evolution of $|W - W_1|$ at $x=1$ for the solution depicted in Fig.~\ref{fig7}.
The parameters of the quasinormal ringdown, $\omega_1\approx 4.197$ and $\gamma_1\approx 0.277$, agree with the values determined independently using linear perturbation theory.}}
\end{figure}

\subsection{Higher codimension attractors}
Thanks to the reflection symmetry $x\rightarrow -x$ of Eq.~\eqref{eq-sx} it is relatively easy to tune initial data to the stable manifolds of static solutions with two and three unstable modes.
For example, consider a one-parameter family of even initial data
\begin{equation}\label{critical-even}
 W(0,x) = -1 + c\, T_2(x), \quad \partial_{\tau} W(0,x) = 0\,.
 \end{equation}
These data interpolate between the basins of attraction of solutions $-W_0$ (for small values of $c$, say $c=0$) and $W_0$ (for large values of $c$, say $c=1$). Using bisection one can determine numerically a critical value $c_*\in [0,1]$ for which the initial data \eqref{critical-even} lie on a borderline between these basins of attraction. It is natural to expect that such data will evolve to the least unstable even solution, i.e. $W_*$ if $\ell<2$ or $W_2$ if $\ell>2$. Numerically $c_*$ cannot be determined exactly, and the corresponding near-critical solution approaches $W_*$ or $W_2$ (depending on the value of $\ell$), stays in its neighbourhood for some time and eventually is ejected out along the one-dimensional unstable manifold (see Figs.~\ref{fig9} and \ref{fig10}). The dynamics of this process can be approximated as follows
\begin{equation}\label{crit-w*}
   W(\tau,x)\sim A_0(c) e^{\ell \tau} + A_2 e^{-(2-\ell) \tau}\, (1-(2\ell-1)x),
\end{equation}
if $\ell<2$, and
\begin{equation}\label{crit-w2}
   W(\tau,x)\sim W_2(x) + A_0(c) e^{\lambda_{2,0} \tau}\, u_{2,0}(x) + A_2 e^{-\gamma_2 \tau}\, \sin(\omega_2 \tau+\delta)\,u_{2,2}(x),
\end{equation}
if $\ell>2$, where $\gamma_2=-\mathrm{Re}(\lambda_{2,2})$ and $\omega_2=\mathrm{Im}(\lambda_{2,2})$. For critical initial data $A_0(c_*)=0$, however in practice there is
always a small admixture of the unstable mode, and therefore $W_*$ and $W_2$ are only intermediate attractors whose lifetimes scale with $c$ as $-\frac{1}{\gamma} \ln|c-c_*|$, where $\gamma=\ell$ in the case of $W_*$ and $\gamma=\lambda_{2,0}$ in the case of $W_2$.

\begin{figure} [h]
\includegraphics[width=\textwidth]{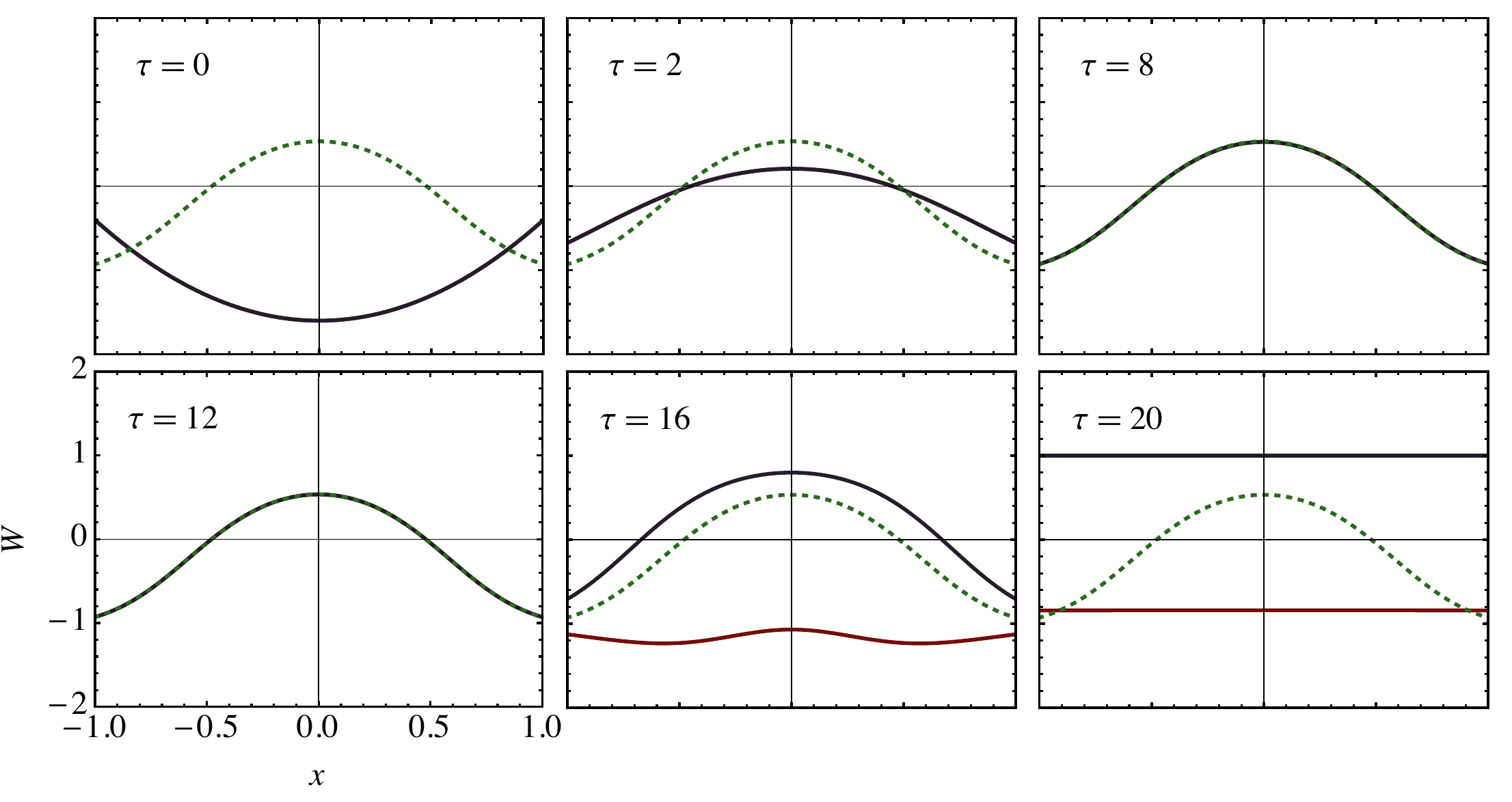}
\captionsetup{width=1.02\textwidth}
\caption{\label{fig9} {\small Snapshots from the evolution of a pair of near critical initial data
(\ref{critical-even}) for $c = c_{-}=0.599712304435598$ and $c~=c_{+}=~0.599712304435599$. By continuous dependence on initial data, these two solutions  evolve for some time close together approaching the unstable attractor $W_2$ (depicted with the dotted line) and eventually tend to $-1$ and $+1$, respectively.}}
\end{figure}
\begin{figure} [h]
\includegraphics[width=0.48\textwidth]{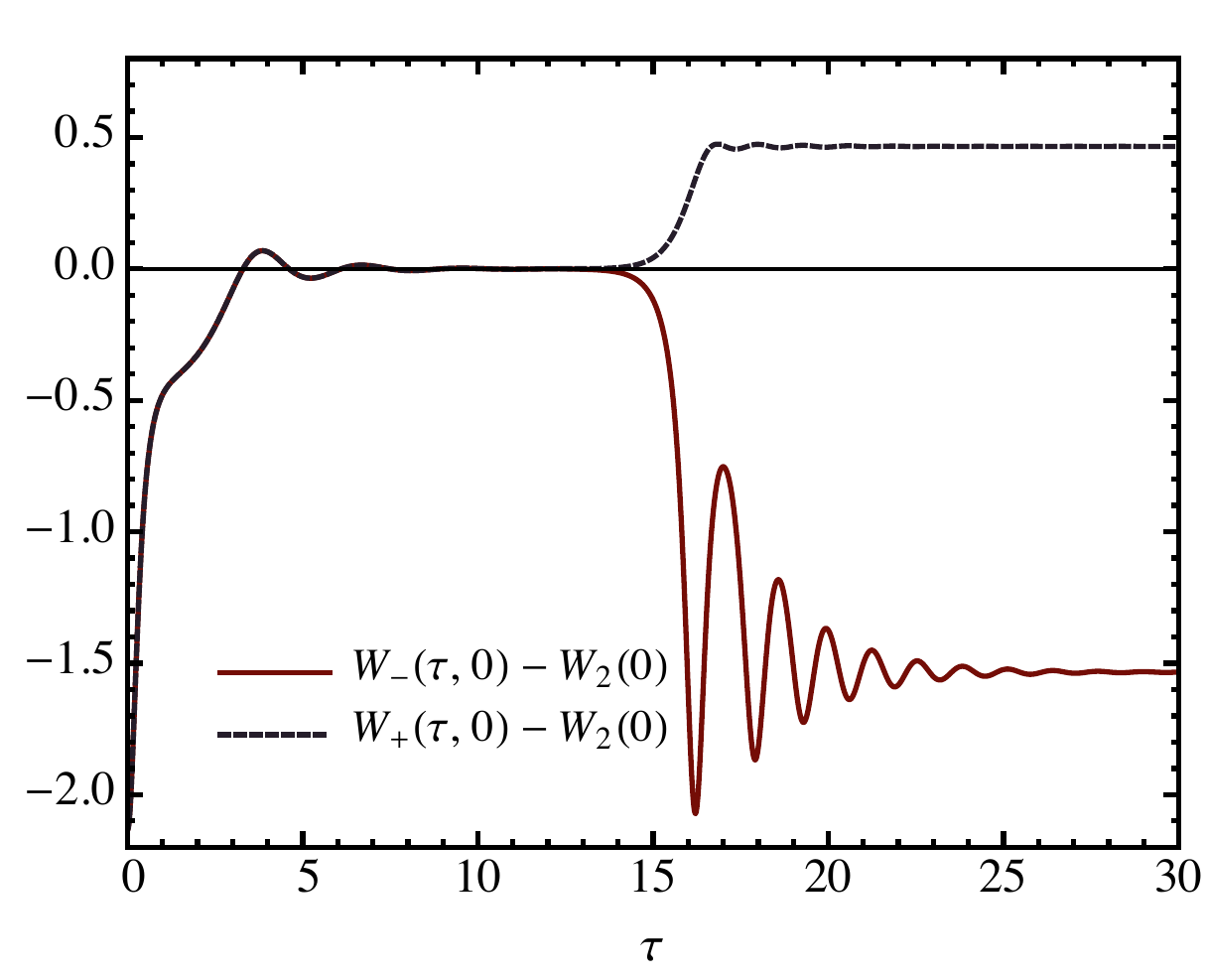}
\includegraphics[width=0.48\textwidth]{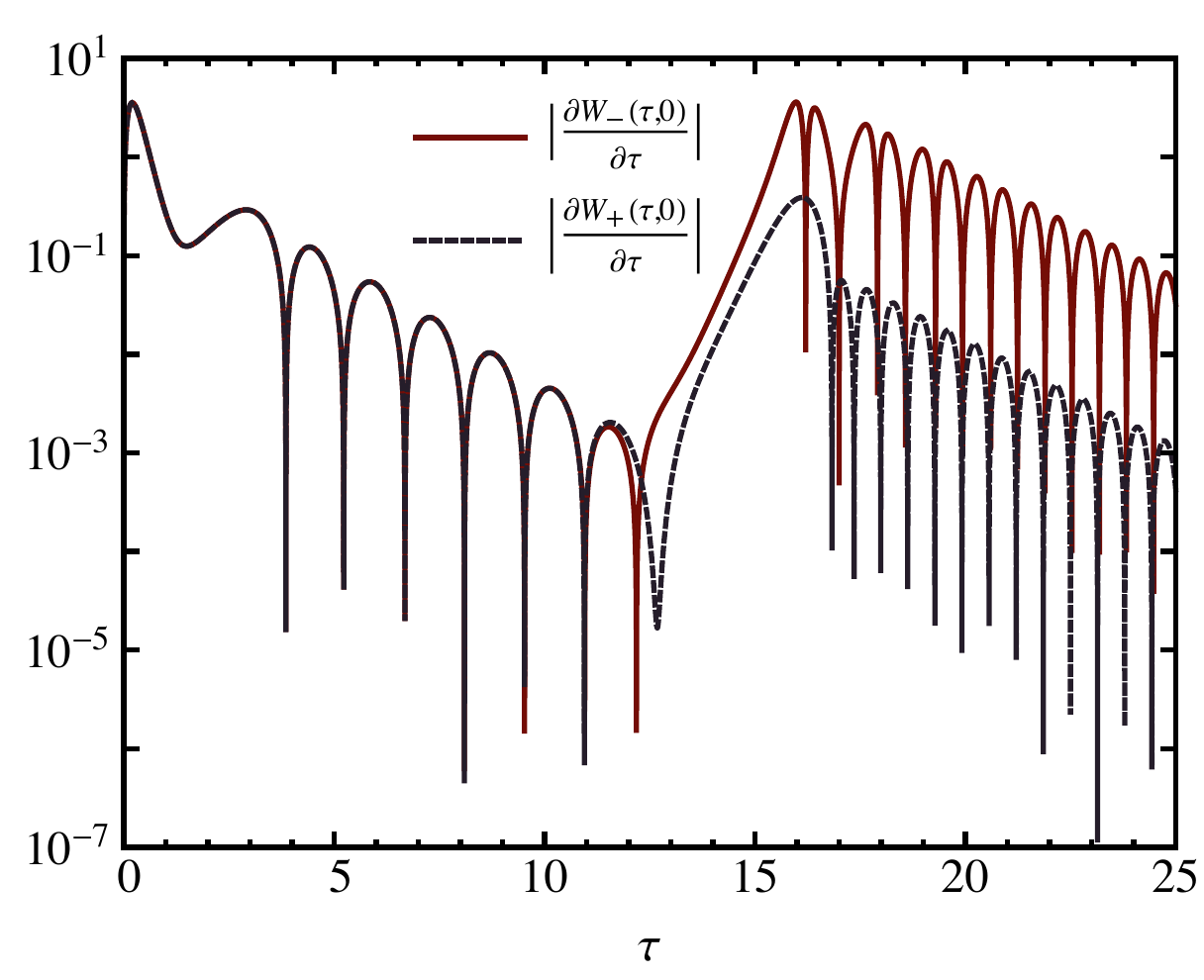}
\captionsetup{width=\textwidth}
\caption{\label{fig10} {\small The solutions corresponding to near critical parameters $c_{\pm}$, depicted in Fig.~\ref{fig9},  are denoted by $W_{\pm}$. The left and right panels show the time evolution of  $W_{\pm} - W_2$ and $|\partial_\tau W_{\pm}|$ at $x = 0$, respectively. One can clearly distinguish three phases of evolution: the approach to the unstable attractor $W_2$ governed by its fundamental quasinormal mode, the departure from $W_2$ governed by its unstable mode, and finally relaxation to $\pm 1$ governed by its fundamental quasinormal mode.}}
\end{figure}

In a similar manner,  one can prepare odd initial data lying on the stable manifold of $W_*$ for  $2<\ell<3$ or $W_3$ for $\ell>3$.

\section{Final remarks}
The study presented above is part of a bigger project which could be called ``Designer PDEs.'' In this project we play with domains of nonlinear PDEs (mainly wave equations) in order to design toy models for studying various physical phenomena
(such as relaxation to equilibrium, weak turbulence, blowup, etc.) in the \textit{simplest} possible settings. The studies of such toy models not only help to understand the underlying phenomena but, in addition, reveal interesting interactions between the geometry of a domain and a nonlinearity. For example, the comparison of dynamics of the Yang-Mills field on the asymptotically hyperbolic wormhole (described in this paper) and the asymptotically flat wormhole (described in the parallel paper \cite{bk})
is very instructive in showing how the late time behavior of waves is affected by the asymptotic behavior and the conformal structure of spacetime.

  Although Eq.~\eqref{eqym} is very special, we believe that it is representative for a larger class of models in the sense that many of its features, discussed above, are structurally stable with respect to perturbations of the geometry of a domain. We hope that our toy model will serve as  a playground for developing new mathematical methods (in particular,  tools based on hyperboloidal foliations)   for semilinear wave equations on asymptotically hyperbolic spacetimes.

\subsubsection*{Acknowledgments.} PB thanks Helmut Friedrich for helpful discussions.  PM acknowledges the hospitality of the A. Einstein Institute in Golm where this work was initiated. This research was supported in part by the NCN grant NN202 030740 and the Polish Ministry of Science and Higher Education grant 7150/E-338/M/2014.

\section*{Appendix: Fourier-Galerkin method}

To solve Eq.~\eqref{eq-sx} we use a version of the spectral Fourier-Galerkin method.
The solution is expanded in the basis of Chebyshev polynomials of the first kind $\{ T_n \}_{n=0,1,\dots}$ as
\begin{equation}
\label{ex}
W(\tau,x) = \sum_{n = 0}^{\infty} a_n(\tau) T_n(x).
\end{equation}
The choice of Chebyshev polynomials is suggested by the form of the Chebyshev equation
\[ (1 - x^2) T_n^{\prime \prime}(x) - x T_n^\prime(x) + n^2 T_n(x) = 0. \]
The other obvious choice would be to work with Legendre polynomials. They seem to be even better adjusted to our needs, since the linear part of the right hand side of Eq.~(\ref{eq-sx}) has the form of the Legendre equation. The reason why we do not work with Legendre polynomials is that the formulae expressing the products of Legendre polynomials as their linear combinations are relatively complicated (cf.~\cite{neumann, adams}). The advantage of using Chebyshev polynomials is that the product of two Chebyshev polynomials can be expressed as
\[ T_m (x) T_n (x) = \frac{1}{2} \left( T_{m+n} (x) + T_{|m - n|} (x) \right), \]
and thus the triple product has the form
\begin{eqnarray*}
T_l (x) T_m (x) T_n (x) & = & \frac{1}{4} \left( T_{l+m+n} (x) + T_{|l - m - n|} (x) + T_{l + |m - n|} (x) \right. \\
&& \left. + T_{|l - |m - n||} (x) \right).
\end{eqnarray*}
Using the above simple formula, we can expand the nonlinear term $W^3$ in Eq.~(\ref{eq-sx}) as follows
\begin{eqnarray} \label{T3}
\left( W(\tau,x) \right)^3 & = & \sum_{l,m,n=0}^{\infty} a_l(\tau) a_m (\tau) a_n (\tau) T_l (x) T_m (x) T_n (x) \nonumber \\
&= & \sum_{n=0}^{\infty} w_n(\tau) T_n (x),
\end{eqnarray}
where
\begin{eqnarray}\label{w0}
w_0 (\tau)  &:= & \frac{1}{4} \left(a_0(\tau) \right)^3 + \frac{1}{4} \sum_{m=0}^{\infty} a_0 (\tau) \left(a_m (\tau) \right)^2 \nonumber \\
& + &\frac{1}{4} \sum_{m,n=0}^{\infty} \left(a_{m+n}(\tau) + a_{|m - n|}(\tau) \right) a_m (\tau) a_n(\tau)
\end{eqnarray}
and, for $k > 0$,
\begin{align}\label{wk}
& w_k(\tau)  :=  \frac{1}{4} \sum_{m,n=0}^{\infty} \left(a_{k - m - n}(\tau) + a_{k - |m - n|}(\tau) + a_{k + m + n}(\tau) \right.
 \nonumber \\
  & +  \left. a_{-k + m + n}(\tau) + a_{k + |m - n|}(\tau) + a_{-k + |m - n|}(\tau) \right) \,a_m(\tau)  a_n(\tau).
\end{align}
Here it is assumed implicitly that all coefficients with negative indices vanish identically.

\begin{figure}[ht]
\includegraphics[width=0.48\textwidth]{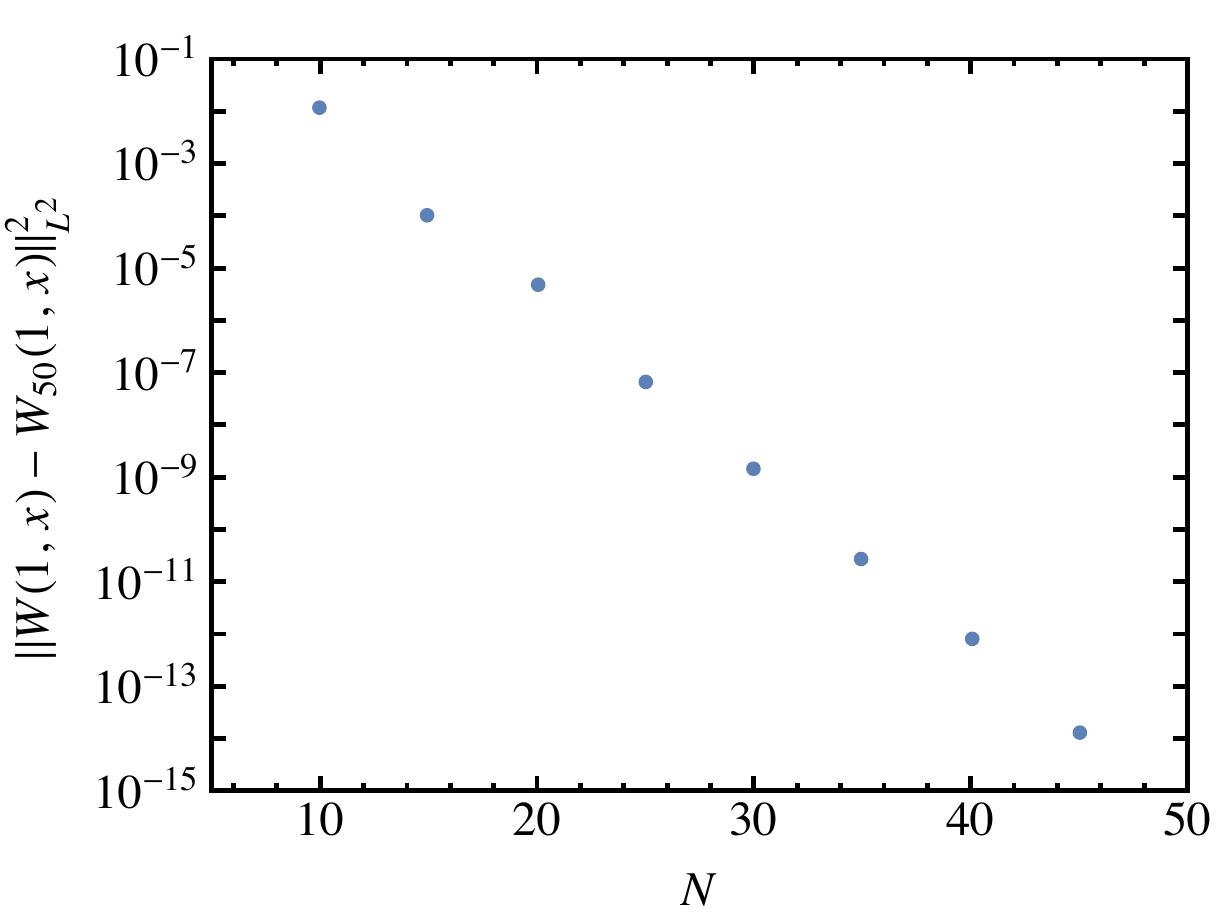}
\captionsetup{width=\textwidth}
\caption{\label{fig11} {\small A sample convergence test of our numerical scheme. The abscissa shows the index $N$ of the highest Chebyshev polynomial used in the computation. The ordinate shows the square of the $L^2([-1,1])$ norm of the difference between two solutions at time $\tau = 1$: the numerical solution computed with the given number of Chebyshev polynomials and the same solution computed with $N = 50$ (this solution is denoted as $W_{50}$). The initial data in this example are given by Eq.~(\ref{ic-odd}).}}
\end{figure}

In order to deal with the terms involving $x \partial_x W$ and $x \partial_{\tau x} W$ in Eq.~(\ref{eq-sx}) we recall that
\[ x T_n^\prime (x) = \left\{ \begin{array}{ll}
n \left( 2 \sum_{j=0}^{n/2} T_{2j} (x) - T_0(x) - T_n(x) \right), & n = 0,2,4,\dots,\\
n \left( 2 \sum_{j=0}^{(n-1)/2} T_{2j+1} (x) - T_n (x) \right), & n = 1,3,5,\dots
\end{array} \right. \]
Thus, defining
\[ z_0(\tau) := \sum_{j=0}^{\infty} 2j a_{2j}(\tau) \]
and
\[ z_k(\tau) := k a_k(\tau) + 2 \sum_{j=1}^{\infty} (k+2j) a_{k+2j}(\tau),  \quad k > 0, \]
one can write
\begin{equation}\label{xTprime}
 x \sum_{n=0}^{\infty} a_n(\tau) T^\prime_n (x) = \sum_{n = 0}^{\infty} z_n(\tau) T_n (x).
 \end{equation}

Inserting expansion (\ref{ex}) into Eq.~(\ref{eq-sx}) and using \eqref{T3} and \eqref{xTprime}, we obtain an infinite  system of ordinary differential equations of the form
\begin{equation}
\label{coeffeq}
\ddot a_n + \dot a_n + (n^2 - \ell (\ell + 1)) a_n + z_n + 2 \dot z_n + \ell (\ell + 1) w_n = 0\,
\end{equation}
for $n=0,1,\dots$ In numerical computations we truncate this infinite system  at some prescribed $n=N$ and retain the first $N+1$ coefficients $a_0,a_1,\dots,a_N$.
 The truncated finite-dimensional  dynamical system  can be solved using standard numerical methods.  This procedure is usually referred to as the Galerkin method. The initial data consist of a set of $2N + 2$ values $\{ a_n(0), \dot a_n(0) \}_{n = 0, \dots, N}$.   In the computations presented in Section~5  we set $N$ of the order of 30 to 50.

The convergence of our spectral numerical scheme is, as expected, exponential, at least in the cases that we have tested so far. An example of the convergence properties (with respect to the increasing number on Chebyshev polynomials used in the computation) is shown in Fig.~\ref{fig11}.

All numerical examples shown in this paper were computed with \textit{Wolfram Mathematica}. A sample \textit{Mathematica} notebook containing the implementation of the above method is attached to this paper.

\end{document}